\documentclass{amsart}

\newtheorem{theorem}{Theorem} 

\newtheorem{lemma}[theorem]{Lemma}
\newtheorem{claim}[theorem]{Claim}
\newtheorem{proposition}[theorem]{Proposition}  
\newtheorem*{problem}{Problem}                                           
\newtheorem*{conjecture}{Conjecture}                

\theoremstyle{definition}
\newtheorem{definition}[theorem]{Definition}

\theoremstyle{remark}
\newtheorem{remark}[theorem]{Remark}

\numberwithin{equation}{section}
\numberwithin{theorem}{section}   

\newcommand{\beq}{\begin{equation}}
\newcommand{\eeq}{\end{equation}}

\newcommand{\ol}[1]{\overline{#1}}                                            
 \newcommand{\mc}[1]{\mathcal{#1}}                                             

\newcommand{\bN}{\mathbb{N}}   
\newcommand{\bR}{\mathbb{R}} 
\newcommand{\bC}{\mathbb{C}} 

\newcommand{\calI}{\mathcal{I}}
\newcommand{\calJ}{\mathcal{J}}
\newcommand{\calK}{\mathcal{K}}
\newcommand{\calL}{\mathcal{L}}
\newcommand{\calM}{\mathcal{M}}
\newcommand{\calS}{\mathcal{S}}

\newcommand{\ca}{{\overline{a}}}
\newcommand{\cb}{{\overline{b}}}
\newcommand{\cc}{{\overline{c}}}
\newcommand{\cd}{{\overline{d}}}
\newcommand{\ce}{{\overline{e}}}
\newcommand{\cs}{{\overline{s}}}
\newcommand{\ct}{{\overline{t}}}

\renewcommand{\Im}{\operatorname{Im}}
\newcommand{\Ric}{\operatorname{Ric}}
\newcommand{\contr}{\operatorname{contr}}
\newcommand{\pcontr}{\operatorname{pcontr}}

\newcommand{\Td}{\operatorname{Td}}
\newcommand{\Local}{\operatorname{Local}}

\newcommand{\ch}{\mathit{ch}}
\newcommand{\Ch}{\mathit{Ch}}
\newcommand{\Tch}{\mathit{Tch}}

\newcommand{\sfR}{\mathsf{R}}

\newcommand{\pa}{\partial}
\newcommand{\wt}{\widetilde}
\newcommand{\wh}{\widehat}
\newcommand{\La}{\Lambda}
\newcommand{\gw}{w_{G}}

\newcommand{\ii}{\sqrt{-1}}

\newcommand{\Special}{\Psi^1_{\ca}\Psi^\sigma_a}
\newcommand{\Specialonek}{\Psi^1_\ca\Psi^k_a}
\newcommand{\Specialks}{\Psi^k_{\ca}\Psi^\sigma_a}
\newcommand{\erase}{\,\backslash\!\!\!\!}

\newcommand\YDbox{
\setlength{\unitlength}{0.15mm}
\begin{picture}(21,21)(1,2)
\put(0,0){\line(0,1){20}}
\put(10,0){\line(0,1){20}}
\put(20,0){\line(0,1){20}}
\put(0,0){\line(1,0){20}}
\put(0,10){\line(1,0){20}}
\put(0,20){\line(1,0){20}}
\end{picture}}

\begin{document}
\title[Integral K\"ahler Invariants]%
{Integral K\"ahler Invariants and \\
the Bergman kernel asymptotics for  line bundles}

\author{Spyros Alexakis}
\address{Department of Mathematics\\ 
University of Toronto\\
40 St George Street Rm 6290
\\ Toronto ON M5S 2E4\\ Canada}
\email{alexakis@math.utoronto.ca}

\author{Kengo Hirachi}
\address{
Graduate School of Mathematical Sciences\\ 
The University of Tokyo\\
3-8-1 Komaba, Megro, Tokyo 153-8914 \\ Japan}
\email{hirachi@ms.u-tokyo.ac.jp}

\begin{abstract}
On a compact K\"ahler manifold, one can define global invariants by integrating local invariants of the metric.  Assume that a global invariant thus obtained depends only on the K\"ahler class. Then we show that the integrand can be decomposed into a Chern polynomial (the integrand of a Chern number) and divergences of one forms, which do not contribute to the integral. We apply this decomposition formula to describe the asymptotic  expansion of the Bergman kernel for positive line bundles and to show that the CR $Q$-curvature on a Sasakian manifold is a divergence.
\end{abstract}



\maketitle
\tableofcontents

\section{Introduction} 
\label{sec:intro}
\subsection{Statement of the main theorem}
A basic problem in differential geometry is to find relations between the local invariants of a geometric structure and the global ones.   Many  important examples of such correspondences arise from the asymptotic analysis of kernel functions:

\smallskip\noindent
(1) 
On Riemannian manifolds, the coefficients of the heat kernel asymptotic expansion give local invariants of the metric, whose integrals are spectral invariants of the Laplacian; see \cite{Gi}. If one uses the Yamabe Laplacian, then one gets a global conformal invariant by the integral, called the conformal index; see \cite{PR}, \cite{BO}.

\smallskip

\noindent
(2)
On polarized K\"ahler manifolds $(X,\calL)$, the Bergman kernel $B_m$ for the sections of $\calL^m$ has an asymptotic expansion as $m\to\infty$, which is know as the Tian-Yau-Zelditch expansion; see \S\ref{asymB} below.  The coefficients of this are local invariants of the metric that integrate to the Chern numbers; the expansion  can bee seen as a local version of the Hirzebruch-Riemann-Roch theorem. 
\smallskip

In these expansions, the explicit computation of the local invariants is not easy; see \cite{En},  \cite{HKN2}, \cite{Lo} and \cite{Lu} for the case (2).  While Weyl's invariant theory for the structure group can be used to simplify the task, it does not reveal the information contained in the integrals of the local invariants.  Thus a natural question is how far we can determine a local invariant from the fact that its integral is a global invariant; especially when the integral depends only on a class of metrics, e.g., a conformal class or K\"ahler class.  A result in the conformal case has been obtained by the first author; see \S\ref{conformal-intro} below.  In the case of K\"ahler geometry, the problem can be formulated as follows:

\begin{problem}
Let $P(g)$ be a scalar-valued local invariant of a K\"ahler metric. Suppose that the integral 
\beq\label{int-eq0}
\int_{X}P(g)dV_g
\eeq
depends only on the K\"ahler class of $g$ for any compact K\"ahler manifold $(X,g)$ of dimension $n$.  Identify $P(g)$ modulo the divergence of a one-form valued local K\"ahler invariant.
\end{problem}

Let us explain the terminology used here; see \S\ref{sec:results} for a more detailed explanation.  A (scalar-valued) local K\"ahler invariant is a polynomial expression $P(g)$ in the metric $g_{a\cb}$, its coordinate derivatives and $(\det g_{a\cb})^{-1}$
which remains invariant under holomorphic changes of coordinates. It is  known that such an invariant polynomial  can be written (non-uniquely) as a linear combination of Weyl invariants, that is, complete contractions of the form
\beq
\label{contraction1}
\contr\big(\nabla^{(p_1,q_1)}R\otimes\cdots\otimes\nabla^{(p_\sigma,q_\sigma)}R\big),
\eeq
where $R$ is the curvature tensor of the metric and $\nabla^{(p,q)}R$ is its iterated covariant derivative of type $(p,q)$.  Note that we can regard the complete contractions as formal expressions; thus a linear combination of \eqref{contraction1} gives local K\"ahler invariants in all dimensions.

A one form-valued local K\"ahler invariant is a linear combination of partial contractions of $\nabla^{(p_1,q_1)}R\otimes\cdots\otimes\nabla^{(p_\sigma, q_\sigma)}R$ that leaves one (holomorphic or anti-holomorphic) index free. Using abstract index notation, we denote such invariants by $T_a(g)$ or  $T_\ca(g)$.  By  Stokes' theorem,  the divergence of these one forms 
$\nabla_\ca T_a(g)$ and $\nabla_a T_\ca (g)$ integrates to zero on compact manifolds (the pair of holomorphic and antiholomorphic indices are assumed to be contracted by the metric).  Thus $\nabla_\ca T_a(g)$ and $\nabla_a T_\ca (g)$ are
trivial examples of $P(g)$ satisfying \eqref{int-eq0}.

Nontrivial examples of $P(g)$ satisfying \eqref{int-eq0} are given by Chern--Weil homomorphisms.  For each $\operatorname{Ad}_{U(n)}$-invariant polynomial on the Lie algebra of $U(n)$, we obtain a local K\"ahler invariant $\Ch(g)$, which has a formal expression independent of the dimension $n$,  that integrates to a Chern number of the 
holomorphic tangent bundle $T^{1,0}X$.   We call such an invariant polynomial of the curvature a {\em Chern polynomial}.  Note that a Chern polynomial of homogenous degree $\sigma$ vanishes when evaluated on a manifold of dimension $n<\sigma$.

To state our main theorem, we need one more definition. For a complete contraction of the from \eqref{contraction1}, we define its {\em geometric weight} by
$$
\gw=p_1+\cdots+p_\sigma+\sigma.
$$
We say $P(g)$ has geometric weight $\gw$ if each of the terms in the linear combination has geometric weight $\gw$.  (The common definition of the weight would be $-2\gw$, but we prefer to make the  geometric weight positive.)  This definition is independent of the expression of $P(g)$ as a linear combination of complete contractions. In fact, $P(g)$ has geometric weight $\gw$ if and only if $P(\lambda g)=\lambda^{-\gw}P(g)$ holds for any $\lambda>0$.  In particular, if a complete contraction does not contain covariant derivatives, which is true for the Chern polynomials, the geometric weight agrees with the degree  $\sigma$.

\medskip
\noindent
{\bf Main Theorem.}{\em\  Let $P(g)$ be a local K\"ahler invariant of geometric weight $\gw$ satisfying the assumption of the problem in a dimension $n\ge \gw-1$. Then there exist a Chern polynomial $\Ch(g)$ and  one-form valued local K\"ahler invariants
$T_{a}(g)$ and $T_{\ca}(g)$ such that}
\beq\label{main-thm-eq}
P(g)=\Ch(g)+\nabla_\ca T_{a}(g)+\nabla_a T_{\ca}(g)
\eeq
holds in all dimensions.

\medskip

The assumption on the geometric weight is necessary as the product of a Chern polynomial $c_{n+1}(g)$ of degree $n+1$  and the scalar curvature $S_g$ has geometric weight $n+2$ and vanishes in dimension $n$, but $c_{n+1}(g)S_g$ does not admit a decomposition like \eqref{main-thm-eq}.

In this theorem, we assume that $P(g)$ has an expression with homogeneous geometric weight.  We do not lose any generality by this since any K\"ahler invariant can be decomposed into the ones with homogeneous geometric weight by keeping the required property of the integral; see \S\ref{fmtsec}.

\subsection{Asymptotic expansion of the Bergman kernel}\label{asymB} We apply the main theorem to describe the asymptotic expansion of the Bergman kernel.  Let us recall the set up. A polarized manifold is a complex manifold $X$ of dimension $n$ with a positive hermitian line bundle $(\calL,h)$  over $X$;  the curvature of $h$ gives a K\"ahler form $\omega$ and an associated metric $g$ on $X$. For each integer $m$, the Bergman kernel of $H^0(X,\calL^m)$ is defined by taking an $L^2$ orthogonal basis $\varphi_{1},\dots,\varphi_{d_{m}}$ and forming the sum
$$
B_{m}(z)=\sum_{j=1}^{d_{m}}\|\varphi_{j}(z)\|^{2}.
$$
It is shown by Catlin \cite{Ca} and Zelditch \cite{Ze} (based on the works  \cite{Ya},  \cite{Ti},  \cite{Bouche}, \cite{Ru}) that $B_m$ has an asymptotic expansion 
$$
B_{m}\sim \sum_{j=0}^{\infty} a_{j}\,m^{n-j}\quad \text{as }m\to \infty,
$$
where $a_j$ are local invariants of the K\"ahler manifold $(X,g)$.  By the scaling of the metric, one can  see that $a_j$ has geometric weight $j$. The first few terms of the expansion have been explicitly written down (\cite{Lu}, \cite{En}, \cite{Lo}) and some algebraic procedures to compute $a_j$ are known; see \cite{Xu} and the Appendix here. 

On the other hand, we can explicitly compute the integral of $a_j$ via the Hirzebruch-Riemann-Roch theorem.  For the curvature form  $\sfR$ of $g$, the Todd genus form is given by
$$
\Td(\sfR)=\det\left(\frac{\sfR}{e^\sfR-1}\right).
$$
Let $\Td_{j}(\sfR)$ be the type $(j,j)$ component of $ \Td(\sfR)$. Then we may define a local K\"ahler invariant $P_{j}$ by the complete contraction:
$$
P_{j}(R)=\frac{1}{j!}\contr\big(\Td_{j}(\sfR)\big),
$$
which has a formal expression independent of the dimension. Then the  Hirzebruch-Riemann-Roch theorem implies 
$$
\int_{X}a_{j}\,\omega^{n}=\int_{X}P_{j}\,\omega^{n}
$$
for any dimension $n$.  Therefore the main theorem  implies that there exist one-form valued local K\"ahler invariants $T_{a}^{j}(g)$, $T_{\ca}^{j}(g)$,  such that
$$
a_j(g)=P_{j}(g)+\nabla_\ca T^{j}_{a}(g)+\nabla_a T^{j}_{\ca}(g).
$$
Since $a_j$ has geometric weight $j$, we know that the terms of degree $j$ (i.e., complete contractions with $\sigma=j$) agrees with $P_j$.  On the other hand, by using the Bianchi and Ricci identities,  it is easy to see that for each $a_j(g)$  the linear term in the curvature can be written as  a multiple of $\Delta^{j-1}S$, where $\Delta=\nabla_{a\ca}$ is the Laplacian and $S$ is the scalar curvature. The constant can be identified by the computation of the first variation of the kernel function under perturbations of the metric; see \cite[Theorem 3.1]{LT} and Appendix:
\beq\label{ajLu}
a_j(g)=\frac{j}{(j+1)!}\Delta^{j-1}S+\text{(non-linear terms)}.
\eeq
Combining the two results on the top and lowest degree terms, we have
$$
\begin{aligned}
a_1&=P_1=\frac12 S,\\
a_2&=P_2+\frac{1}{3}\Delta S.
\end{aligned}
$$
For $a_3$, the formula of Lu \cite{Lu} can be written as
$$
a_3=P_3+\nabla_\ca Q_a+\frac{1}{8}\Delta S,
$$
where $Q_a$ is a one form given by quadratic terms
$$
48 Q_a=\nabla_a(|R|^2-4|\Ric|^2+8S^2)
+2\nabla_{d} (R_{a\cb c\cd}\Ric_{b \cc}
-4 S \Ric_{a\cd}).
$$
Here $\Ric_{a\cb}=-R_{a\cb c\cc }$ is the Ricci curvature.

It is worth noting that $a_1=S/2$ is a key formula in Donaldson's proof \cite{Do} of the stability of the polarized manifolds $(X,\calL)$ with constant scalar curvature K\"ahler form in the first Chern class of  $\calL$; see Sz\'ekelyhidi
 \cite{Sz} for an introduction to this field.  The study of general $a_j$ can be seen as a part of Fefferman's program initiated in \cite{Fe}, where he proposed to study the Bergman and Szeg\"o kernels on strictly pseudoconvex domains as analogies of the heat kernel expansion on Riemannian manifolds. See \cite{Hi4}, \cite{Hi} and \cite{LT} for some progress in this direction.

\subsection{A comparison to global conformal invariants}
\label{conformal-intro}
We briefly comment on the relation between this work and the analogous problem in conformal geometry. There the issue was to understand the space of Riemannian invariants $P(g)$ in (real) dimension $n$, of geometric weight $n/2$ for which the integral 
\[
\int_MP(g)dV_g
\]
over compact Riemannian  manifolds $(M,g)$ is invariant under conformal changes of the underlying metric. 

In a series of works \cite{A1}--\cite{A5}, the first author showed that $P(g)$ can then be expressed as a sum of a local conformal invariant $W(g)$ of geometric weight $n/2$, a divergence $\nabla_{a}T_a(g)$,  and of the a multiple of the Pfaffian of the curvature tensor (i.e., the Chern-Gauss-Bonnet integrand) $\operatorname{Pfaff}(R)$:

\[
P(g)=W(g)+\nabla_aT_a(g)+c\cdot \operatorname{Pfaff}(R).
\]  
This is thus an analogue of the Main Theorem  above, where invariance under K\"ahler deformations of the metric is replaced with invariance under conformal transformations. 
(We can also say that the change of metric in a K\"ahler class corresponds to the conformal class of the hermitian metric on the line bundle $\calL$.)

The proof of the result in the conformal case is much more involved. This is essentially due to two reasons: Firstly, the very existence of (numerous) local conformal invariants makes the task of proving the result more challenging; in fact a major obstacle is how to \emph{separate} the local conformal invariant $W(g)$ from the divergence $\nabla_aT_a(g)$. Since there exist no non-trivial \emph{local} invariants of the K\"ahler class, one does not have this difficulty here. Secondly, an extra challenge in the conformal case is the algebraic complexity of the underlying local invariants $P(g)$: Indeed the curvature tensor and its covariant derivatives satisfy the symmetries of a Young tableau $\YDbox$, while in the K\"ahler case, once the metric is expressed in terms of the K\"ahler potential, the covariant derivatives of the curvature are (up to nonlinear terms) symmetric.

However, the K\"ahler setting does have an extra difficulty compared to the conformal case, which affects the proof very significantly. K\"ahler invariants are complete contractions of tensors where indices are naturally distinguished into two \emph{types}: holomorphic and anti-holomorphic. Accordingly, the divergences that one seeks to construct must preserve this structure; in that respect this raises difficulties not present in the 
conformal setting, which we now review.

\subsection{An overview of the proof}
The main strategy (as in \cite{A3}) is to proceed by an iteration: We choose the terms in $P(g)$ which have the highest order (when seen as a differential operator on the curvature tensor); if this order is non-zero, we show that that these terms separately can be expressed as a divergence, modulo corrections of lower order. Once we have shown this step, then by an iterative argument we are reduced to the case of order zero. In that case, we show that the invariant must be a  Chern polynomial.

To show the main iterative step, we use the fact that the \emph{variation} $L(\Psi)$ with respect to variations of the K\"ahler potential must always integrate to zero; see \S\ref{variatoin-int} below. This implies quite readily that $L(\Psi)$ is expressible  as a divergence. We refer to the formula we thus obtain as a {\em local divergence formula}.%
\footnote{The analogue of this formula in the conformal case was called the ``silly divergence formula" in \cite{A1}.}
However this does \emph{not} in itself imply that the top-oder terms in $P(g)$ are themselves a divergence. The derivation of this fact occupies the bulk of the present paper. 

The main approach in carrying out this strategy is to ``normalize" the top-order terms in $P(g)$ as much as possible by subtracting divergences. Once the top order terms have been normalized, the local divergence formula (together with an inductive argument on the proposition we are
proving) allow us to show that this normalized piece is again a divergence, thus our main claim follows. It is here that matters are much more involved in the K\"ahler case:

The main insight obtained in \cite{A1} is that if the top-order terms in $L(\Psi)$ do not contain any term with two indices in the same factor contracting against each other\footnote{Such contractions are called ``internal contractions" in \cite{A1} and ``traces" here.} then the top-order term must vanish. 
Proving this is difficult; it relies on the ``super divergence formula" in \cite{A1}. This tool is then put to use in proving the inductive step for the conformal setting. However this statement is \emph{false} in the K\"ahler setting. The existence of indices of two types allows one to easily construct divergences with no  traces. Thus the task of normalizing the top-order terms in $P(g)$ as much as possible becomes much more complicated. In fact, the explicit constructions of divergences done in various parts of \S\ref{dgt0} precisely serve this role of normalizing $P(g)$. The final task is to use the local divergence formula to pick out  a ``piece" of the (normalized) top-order term in $P(g)$. This relies on a new induction that depends on five parameters. 

This paper organized as follows.  In \S\ref{sec:results} we take the variation of the integral \eqref{int-eq0} and reduce the main theorem to a proposition for the variation $L(\Psi)$. 
We prove the proposition in \S\ref{dgt0} and \S\ref{chern_poly} respectively in the case of positive order and order zero.  In Appendix we give an algebraic procedure to compute the asymptotic expansion of the Bergman kernel using the result of \cite{HKN1} by using the Szeg\"o kernel of the circle bundle in  a negative line bundle; here we also apply the main theorem to study the $Q$-curvature of the circle bundle.

\section{Formulation of the problem and an outline of the ideas}
\label{sec:results}

The theorem will be proven by an inductive argument. We introduce some key concepts here that will be used extensively in the whole paper. 

\subsection{Local K\"ahler invariants: the first main theorem in invariant theory}
\label{fmtsec}
The {\em local invariants for a K\"ahler metric $g$} are defined to be polynomial expressions  $P(g)$ in the metric $g_{a\cb}$, its coordinate derivatives and $(\det g_{a\cb})^{-1}$, which remain invariant under holomorphic changes of coordinates. We say that $P(g)$ has {\em  geometric weight $\gw$} if it satisfies 
$$
P(\lambda g)=\lambda^{-\gw} P(g)
$$
for any positive constant $\lambda$.   Any local invariant can be decomposed into a sum of  terms with the same geometric weight:
$$
P(g)=\sum_w P_w(g).
$$
Moreover, if the integral of $P(g)$ is an invariant of the K\"ahler class, then this is also true for each $P_w(g)$; thus we may assume that 
all terms in $P(g)$ have a given geometric  weight,  without loss of generality.  In fact, if
$$
\int P(\wh g)dV_{\wh g}=\int P(g)dV_{g}
$$
whenever $\wh g=g+\ii\pa\ol\pa f$ ($g$ is identified with the corresponding K\"ahler form), then in view of $\lambda\wh g=\lambda g+\ii\pa\ol\pa\lambda f$,
we also have
$$
\int P(\lambda\wh g)dV_{\lambda\wh g}=\int P(\lambda g)dV_{\lambda g}.
$$
Thus expanding both sides in powers of $\lambda$, we obtain
$$
\int P_w(\wh g)dV_{\wh g}=\int P_w( g)dV_{g}
$$
as claimed.

The first main theorem in  invariant theory for the group $U(n)$ (see \cite{Weyl} and \cite{BEG}) shows that any such local K\"ahler invariant $P(g)$ can be expressed (non-uniquely) as a linear combination of  complete contractions in the iterated covariant derivatives of curvature tensor
$$
R^{(A,B)}=\nabla^{(A-2,B-2)}R,
$$
all of which have a given geometric weight. Thus
\beq\label{linear-combination}
P(g)=\sum_{l\in \La} a_l C^l(g),
\eeq
where $\La$ is a finite index set, $a_l$ are constants and each $C^l(g)$ is a complete contraction  constructed as follows: For a list of integers $(A_1,\dots,A_\sigma,B_1,\dots,B_\sigma)$ such that $A_j,B_j\ge2$ and 
$$
A_1+\cdots +A_\sigma=B_1+\cdots+B_\sigma,
$$
we consider a complete contraction of the form
\beq\label{contr1}
C(g)=
\contr\big(R^{(A_1,B_1)}
\otimes\dots\otimes R^{(A_\sigma,B_\sigma)}
\big),
\eeq
where the contraction is taken with respect to $w$ pairings of holomorphic and antiholomorphic indices; so we have $w=A_1+\cdots +A_\sigma$. Note that the geometric weight  of such an invariant is $\gw=w-\sigma$.  

\begin{definition}
For a complete (or partial) contraction of the form \eqref{contr1}, we define its {\em weight},\footnote{Note that the weight differs from the \emph{geometric weight} defined above.}  {\em degree} and  {\em order}, respectively, by $w$ (the number of contractions), $\sigma$ (the number of factors $R^{(A,B)}$) and 
$$
d=\sum_{j=1}^\sigma (A_j+B_j-4).
$$
\end{definition}

Note that  the order $d$ is the total number of the derivatives applied to the curvature. The equality $d=0$ holds if and only if $(A_j,B_j)=(2,2)$ for all $j$, i.e., no derivatives are  applied to the curvature tensor $R^{(2,2)}$.

\begin{definition}
For a linear combination of complete contractions \eqref{linear-combination}, its {\em minimal degree} $\sigma$ is defined to be the  minimum of the degrees of $C^l(g)$, $l\in\Lambda$.  Let $\Lambda^\sigma\subset\Lambda$ be the sub index set for which  $C^l(g)$ has degree $\sigma$.  We then define the {\em sublinear combination consisting of the terms of degree $\sigma$} to be
$$
P^\sigma(g)=\sum_{l\in \Lambda^\sigma} a_l C^l(g).
$$
\end{definition}

The definition of the minimal degree depends on the choice of expression of $P(g)$ as linear combination of complete contractions. However, we can estimate minimal degree $\sigma$  by the geometric weight $\gw$. We alway have
$$
\sigma\le\sigma+\frac{d}{2}= \gw
$$
and the equality holds only when the degree $d=0$.

\subsection{Chern polynomials}\label{chern-poly}
The Chern numbers of the holomorphic tangent bundle $T^{1,0}X$ are given by the integration of local K\"ahler invariants of order  $0$.  We briefly recall the construction.

Let $P(A)$ be an $\operatorname{Ad}_{U(n)}$-invariant polynomial in the components of a skew hermitian matrix $A_{a\cb}$ of homogenous degree $k$. 
Then substituting the curvature form
$$
\mathsf{R}_{a\cb}=R_{a\cb c\cd}dz^c\wedge dz^\cd
$$
into $P(A)$, we obtain  a $(k,k)$-from $P(\mathsf{R})$.  By the Bianchi identity, we see that $P(\mathsf{R})$ is a closed form; moreover, the de Rham class $[P(\mathsf{R})]$ depends only on the K\"ahler class. 
For example, if 
$$
P(A)=A_{a_1\ca_{k}}A_{a_2\ca_{1}}\cdots A_{a_k\ca_{k-1}},
$$
then the corresponding $(k,k)$-form is called {\em $k$-th Chern character} form and is denoted by 
$\ch_k(\mathsf{R})$.

On $n$ dimensional manifold with $n\ge k$, we can define a $(n,n)$-form by
$$
P(\mathsf{R})\wedge\omega^{n-k}/(n-k)!,
$$
whose de Rham class is also determined by the K\"ahler class $[\omega]$. Let us define a local K\"ahler invariant $P(g)$ of degree $k$ by
$$
P(g)\omega^n/n!=P(\mathsf{R})\wedge\omega^{n-k}/(n-k)!.
$$
Contracting both side $n$ times, we  get
$$
P(g)= \frac{\pm 1}{k!}\contr P(\mathsf{R}),
$$
where the sign depends on the paring of the indices in the complete contraction.

This $P(g)$ has weight $w=2k$ and vanishes on any  manifold of complex dimension $n<k$. For $n\ge k$, since the de Rham class  $[P(\mathsf{R})\wedge\omega^{n-k}]$ is determined by $[\omega]$, so is
$$
\int P(g)dV_g.
$$
We call linear combinations of K\"ahler invariants obtained in this way {\em Chern polynomials}.  

\subsection{Variation of local K\"ahler invariants and polarization}
\label{variatoin-int}
We will find it useful here to (locally) think of the K\"ahler metric $g$ in terms of a potential function. 

Consider a complex torus $X=\bC^n/\Gamma$ with flat K\"ahler metric $g_0$ defined by  the K\"ahler form
$$
\omega_0=\sqrt{-1} \pa\ol\pa |z|^2,
$$
where $z=(z^1,\dots,z^n)$ are the standard coordinates on $\bC^n$. We take a small open set $U\subset\bC^n$ and regard $z$ as local coordinates of $X$.  Then for a function $\phi\in C_0^\infty(U)$ and small $\epsilon$,
\beq\label{g-family}
\omega_\epsilon=\omega_0+\epsilon\sqrt{-1}\pa\ol\pa\phi
\eeq
gives a family of K\"ahler metrics $g_\epsilon$ on $X$. We  consider the variation of $P(g_\epsilon)$ with respect to $\epsilon$.  Let $\sigma$ be the minimal degree  of  $P(g_\epsilon)$ and define
$$
L^\sigma(\phi)=\frac{d^\sigma}{d\epsilon^\sigma} \Big|_{\epsilon=0}P(g_\epsilon).
$$
We can explicitly write down $L^\sigma(\psi)$ as follows: Let $P^\sigma(g)$ be the sublinear combination of $P(g)$ consisting of the terms of degree $\sigma$,
$$
P^\sigma(g)=
\sum_{l\in \Lambda} a_l \contr^l\big(R^{(A_1,B_1)}\otimes\cdots\otimes 
R^{(A_\sigma,B_\sigma)}\big).
$$
By formally replacing each factor $R^{(p,q)}$ by a factor $\pa^{(p,q)}\phi$, and then contracting the same pairs of indices, we have
\beq
\label{phi-formI}
L^\sigma(\phi)=\sum_{l\in \Lambda} a_l \contr^l(\pa^{(A_1,B_1)}\phi\otimes\cdots\otimes 
\pa^{(A_\sigma,B_\sigma)}\phi),
\eeq
where the contraction and covariant derivatives are defined with respect to the flat metric $g_0$.   
(For the flat connection we use the notation $\pa^{(A,B)}$ in place of $\nabla^{(A,B)}$.)

We also use the polarization of $L^\sigma(\phi)$, which is denoted by $L^\sigma(\Psi)$, where $\Psi=(\psi^1,\dots,\psi^\sigma)\in C^\infty_0(\bC^n,\bR^\sigma)$.  In terms of the differentials it is given by
$$
L^\sigma(\Psi)
=\frac{1}{\sigma!}
\frac{\pa^\sigma}{\pa\lambda_1\cdots \pa\lambda_\sigma}
L^\sigma(\lambda_1\psi^1+\cdots+\lambda_\sigma\psi^\sigma)
\Big|_{(\lambda_1,\dots,\lambda_\sigma)=0}.
$$
Or, when $L^\sigma(\phi)$ is in the form \eqref{phi-formI}, it is given by the substitution
$$
L^\sigma(\Psi)=\frac{1}{\sigma!}\sum_{(\pi,l)\in S_\sigma\times\Lambda}
a_l \contr^l(\pa^{(A_1,B_1)}\psi^{\pi(1)}\otimes\cdots\otimes 
\pa^{(A_\sigma,B_\sigma)}\psi^{\pi(\sigma)}),
$$
where contractions are taken as in the ones that is indexed by $l\in \Lambda$
and $S_\sigma$ is the symmetric group on $\{1,2,\dots,\sigma\}$. More generally, we also consider complete contractions of $\Psi$ that are linear in each $\psi^j$:
\beq \label{psi-formI}
L(\Psi)=\sum_{l\in \Lambda}
a_l \contr^l(\pa^{(A_1,B_1)}\psi^1\otimes\cdots\otimes 
\pa^{(A_\sigma,B_\sigma)}\psi^\sigma)
\eeq
but may not be symmetric in $(\psi^1,\dots,\psi^\sigma)$.  We call such an $L(\Psi)$ an {\em invariant of $\Psi$ of degree $\sigma$}.

\begin{definition}
\label{accetable}
An {\em acceptable invariant} of $\phi$ (resp. $\Psi=(\psi^1,\dots,\psi^\sigma)$) is a linear combination of the form \eqref{phi-formI} (resp. \eqref{psi-formI}) with 
\beq\label{kahler-restriction}
A_1,B_1,\dots, A_\sigma,B_\sigma\ge2.
\eeq
Similarly we define a {\em  $(p,q)$-tensor acceptable invariant} to be a linear combination of partial contractions of the tensors of the form 
$$
\pa^{(A_1,B_1)}\phi\otimes\cdots\otimes 
\pa^{(A_\sigma,B_\sigma)}\phi
$$
(resp. $\pa^{(A_1,B_1)}\psi^1\otimes\cdots\otimes\pa^{(A_\sigma,B_\sigma)}\psi^\sigma$) with \eqref{kahler-restriction} that leaves type $(p,q)$ free indices.  
\end{definition}

\begin{definition}
The {\em weight} of a partial contractions is defined to be the total number of contractions. We say that an acceptable invariant has weight $w$ if each term in the linear combination has weight $w$. 
\end{definition}

Note that if $P(g)$ has geometric weight $\gw$ and minimal degree $\sigma$, then its variation $L^\sigma(\psi)$ has weight $w=\gw+\sigma$ and degree $\sigma$.

\subsection{The second main theorem of invariant theory}
\label{section-SMT}
We have used the first main theorem in classical invariant theory to show that local invariants of a  K\"ahler metric are generated by complete contractions of the iterated covariant derivatives of curvature tensors.  The  relations among the complete contractions are given by the second main theorem of invariant theory.  Using this theorem, we show that the formal expression of $L^\sigma(\psi)$ is uniquely determined by the functional $C^\infty_0(\bC^n)\ni\phi\mapsto L^\sigma(\phi)\in\bC$ when $\sigma\le n$.

To make a precise statement, let us start by defining the notion of  equivalence for formal expressions.  Let $W=\bC^{n}$ with the standard hermitian metric.
We regard $W$ as the standard representation space of $U(n)$, which acts as left-multiplication on column vectors.  Let $W^*$ be the dual representation and  
$\ol {W^*}$ be its conjugate representation. We then define $U(n)$-modules
$$
V^{p,q}
=\bigotimes^p W^*\otimes
\bigotimes^p \ol{W^*}.
$$
Note that $V^{p,q}$ contains a submodule $S^{p,q}=\bigodot^p W^*\otimes\bigodot^p \ol{W^*}$, where $\bigodot$ denotes the symmetric tensor products.

We consider the $U(n)$-invariant polynomials in the components of the collection of tensors $(u^{(p,q)})\in\calS=\prod_{p,q\ge0}^{\infty}S^{p,q}$. By the first main theorem of classical invariant theory, we know that such an invariant polynomial of homogenous degree $\sigma$ can be expressed as a linear combination of complete contractions:
\beq
\label{general.comb}
L(u)=
\sum_{l\in \La} a_l \contr^l(u^{(A_1,B_1)}\otimes\dots\otimes u^{(A_\sigma,B_\sigma)}).
\eeq
We regard two formal complete (or partial) contractions of the form 
$\contr(u^{(A_1,B_1)}\otimes\dots\otimes u^{(A_\sigma,B_\sigma)})$ 
as the same if the ordered list $u^{(A_1,B_1)}, \dots, u^{(A_\sigma,B_\sigma)}$ and the pairings can be made to coincide by permuting the tensors $u^{(A_j,B_j)}$ and by permuting the barred and unbarred indices on each of the tensors.  We say that a linear combination of such formal complete 
(or partial) contractions {\em vanishes formally} if it can be made the same as the zero linear combination, by applications of the operations above and the distributive rule.

Given $L(u)$, we may define a differential operator $L(\phi)$ by substituting $\pa^{(A_j,B_j)}\phi$ into $u^{(A_j,B_j)}$.  If we fix a point $z_0\in\bC^n$, then $\pa^{(A_j,B_j)}\phi(z_0)$ runs through $S^{A_j,B_j}$ as $\phi$ varies in $C^\infty_0(\bC^n)$.  Thus the operator
$$
C^\infty_0(\bC^n)\ni\phi\mapsto L(\phi)\in V^{p,q}
$$ 
vanishes identically if and only if $L(u)=0$ for any $u\in\calS$.  In this case, we say that $L(u)$ {\em vanishes by substitution in dimension $n$}.

Clearly, if $L(u)$ vanishes formally then it must also vanish by substitution. The second main theorem of invariant theory shows us that the converse is also true,  provided the dimension $n$ is larger than the degree of $L(u)$; see \cite[Theorem C.3]{BEG} for the proof.

\begin{theorem}\label{second_main_thm}
Let $L(u)$ be a linear combination of complete (or partial) contractions \eqref{general.comb} of  degree $\sigma$. Assume that there exists an $n\ge\sigma$ such that $L(u)$ vanishes by substitution in dimension $n$. Then  $L(u)$ vanishes formally. 
\end{theorem}

Note that the condition $n\ge\sigma$ is sharp as a Chern polynomial of degree  $\sigma$ vanishes by substitution in dimension $n<\sigma$.

We also use this theorem in the setting of multilinear invariants.   We consider linear combinations of complete (or partial) contractions of the form
\beq
\label{polar.comb}
L(u_1,\dots,u_\sigma)=
\sum_{l\in \La} a_l \contr^l\big(u_1^{(A_1,B_1)}\otimes\dots\otimes u_\sigma^{(A_\sigma,B_\sigma)}\big).
\eeq
$L$ is linear in each $u_j\in\calS$ and the degree  is defined to be $\sigma$.  We regard two formal complete (or partial) contractions of this  form as the same if the pairings of indices can be made to coincide by permuting the barred and unbarred indices on each of $u_j^{(A_j,B_j)}$; accordingly we may define the notion of $L(u_1,\dots,u_\sigma)$ vanishing formally.

By substitution of $\pa^{(A_j,B_j)}\psi^j$ into $u_j^{(A_j,B_j)}$, we may now define a multilinear differential operator
$$
C^\infty_0(\bC^n,\bR^\sigma)\ni\Psi=(\psi^1,\dots,\psi^\sigma)\mapsto L(\Psi)\in V^{p,q}.
$$
If this  operator vanishes identically, we say that $L(u_1,\dots,u_\sigma)$  vanishes by substitution in dimension $n$. With these definitions, Theorem \ref{second_main_thm} also holds  for $L(u_1,\dots,u_\sigma)$. Actually, in \cite{BEG}, Bailey, Eastwood and Graham  proved the theorem in the multilinear case and used the polarization to imply Theorem \ref{second_main_thm}.

\subsection{Main Proposition}
Now we consider the integral equation for the invariant $L^\sigma(\phi)$.  Suppose that $P(g)$ is a local K\"ahler invariant such that the integral 
$$
\int_X P(g)dV_g
$$
depends only on the K\"ahler class of $g$ for any compact K\"ahler manifold of dimension $n$.  Then considering the family of metrics $g_\epsilon$ given by
\eqref{g-family}, we obtain
\beq\label{int-zero-I}
\int_{\bC^n} L^\sigma(\phi)dV=0\quad \text{for any }\phi\in C_0^\infty(\bC^n),
\eeq
where $dV$ is the standard volume form on the Euclidean space $\bC^n=\bR^{2n}$.
(We have assumed $\phi\in C_0^\infty(U)$ in \eqref{g-family}, but $U$ can be taken to be any bounded open set by choosing the lattice $\Gamma$ properly.)

\begin{definition}
We say that $L^\sigma(\phi)$ {\em  integrates to zero in dimension $n$} if \eqref{int-zero-I} holds. This definition can be generalized for the invariant $L(\Psi)$ of $\Psi=(\psi^1,\dots,\psi^\sigma)$, where $\phi\in C^\infty_0(\bC^n)$ is replaced by $\Psi\in C^\infty_0(\bC^n,\bR^\sigma)$.
\end{definition}

If we take $P(g)$ to be a Chern polynomial of homogeneous degree $\sigma$, then the $\sigma$-th variation gives $L^\sigma(\phi)$ that integrates to zero in all dimensions. We also call $L^\sigma(\phi)$ a {\em Chern polynomial.}

We next introduce the divergence of invariants of $\phi$. Let $T_a(\phi)$ be a $(1,0)$-from valued acceptable invariant of $\phi$:
$$
T_b(\phi)=\sum_{l\in\Lambda}a_l\pcontr^l_b\big(\pa^{(A_1,B_1)}\phi\otimes\cdots
\otimes\pa^{(A_\sigma,B_\sigma)}\phi\big),
$$
where the free index $b$ is contained in one of the factors in each term.  The divergence of $T_b(\phi)$ is  defined by
$$
\pa_{\cb}T_b(\phi)=\sum_{l\in\Lambda}a_l\pa_{\cb}\pcontr^l_b\big(\pa^{(A_1,B_1)}\phi\otimes\cdots
\otimes\pa^{(A_\sigma,B_\sigma)}\phi\big),
$$
where $\cb$ and $b$ are contracted with respect to the flat metric $g_0$. The derivative $\pa_\cb$ in each term can be expanded by using the Leibnitz rule and we obtain a sum of $\sigma$ complete contractions.  Note that each term has one more derivative in a factor and one additional contraction.  Thus if
$T_b(\phi)$ is acceptable and has weight $w$, then  $\pa_{\cb}T_b(\phi)$ is also acceptable and has weight $w+1$.

We can also define the divergence for $(0,1)$-from $T_{\cb}(\phi)$ by $\pa_b T_{\cb}(\phi)$.

\begin{proposition}\label{mainprop}
Let $L(\phi)$ be an acceptable invariant of degree $\sigma$, weight $w$ and  order $d$.  Assume that for an  $n\ge \sigma-1$, $L(\phi)$ integrates to zero in dimension $n$.

\medskip
\noindent
{\rm (a)} If $d>0$, then  there exist 1-form valued acceptable invariants $T_a(\phi)$ and $ T_{\ol{a}}(\phi)$ of degree $\sigma$ and weight $w-1$ such that
\beq
\label{theclaim}
L(\phi)=\pa_\ca T_a(\phi)+\pa_a T_\ca(\phi);
\eeq
moreover the above holds formally.

\smallskip
\noindent
{\rm (b)} 
If $d=0$, then $L(\phi)$ is a Chern polynomial.
\end{proposition}

We can easily reduce our main theorem to this proposition.

\begin{proof}[Proof of the main theorem by using Proposition \ref{mainprop}]
Let $\sigma$ be the minimal degree of $P(g)$. Then $n\ge\gw-1$ implies $n\ge \sigma-1$ and thus the variation $L^\sigma(\phi)$ of $P(g)$ satisfies the assumption of Proposition \ref{mainprop}. If $\gw>\sigma$, then $L^\sigma(\phi)$ has order
$$
d=2(\gw-\sigma)>0.
$$
Hence Proposition \ref{mainprop} (a) gives $T_a(\phi)$ and $T_\ca(\phi)$ satisfying \eqref{theclaim}.
  Substituting $R^{(p,q)}$ into $\pa^{(p,q)}\phi$, we define one-from valued local K\"ahler invariants $T_a(g)$, $T_\ca(g)$. Now, since \eqref{theclaim} holds formally, we  repeat the formal operations by which  the left-hand side  of \eqref{theclaim} is made identical to its right-hand side. We derive that
$\pa_\ca T_a(g)+\pa_a T_\ca(g)$ is equal (by substitution) to sum of terms of degree $\sigma$ in $P(g)$,  modulo correction terms of degree $\sigma+1$ and geometric weight $\gw$.   Thus we derive  that $P'(g)$ defined via
$$
P'(g)=P(g)-\pa_\ca T_a(g)-\pa_a T_\ca(g)
$$
has degree $>\sigma$, the geometric weight $\gw$, and the integral
$$
\int_X P'(g)dV_g=\int_X P(g)dV_g
$$
depends only on the K\"ahler class for any compact K\"ahler manifold $(X,g)$ of dimension $n$.   We can repeat this procedure and raise the minimal degree until we get $d=0$. When  $d=0$,   Proposition \ref{mainprop} (b) shows that $L^\sigma(\phi)$ is a Chern polynomial.  In this case, the substitution of $R^{(2,2)}$ into $\pa^{(2,2)}\phi$  recovers $P(g)$ and we see that $P(g)$ is also a Chern polynomial.
\end{proof}

\subsection{Some definitions, notations and  tools}\label{tools}
The rest of this paper is devoted to the proof of Proposition \ref{mainprop}. We now introduce two important  tools that will be repeatedly used in the proof.
The first  we call the ``local divergence formula'';  this is a collection of explicit formulae that  express a local K\"ahler invariant $L(\Psi)$ that always integrates to zero as a  divergence. The next tool is a technique to {\it re-obtain} a new integral equation,  from the local divergence formula, after applying simple algebraic manipulations to the ``local divergence formula'' of $L(\Psi)$.

\subsubsection{Definitions and notations}
\label{Notation-subsec}
We introduce some  notations to simplify the computations involving  contractions.  We will set 
$$
\Psi^j=\pa^{(A_j,B_j)}\psi^j
$$
and write the indices as
$$
\Psi^j_{t_1\dots t_a\cs_1\dots \cs_b}=\pa_{t_1\dots t_a\cs_1\dots \cs_b}\psi^j,
$$
where $a=A_j$ and $b=B_j$.  We also use multi index notation, e.g.,
$$
\Psi^j_{\calI\calJ},
\quad\text{
where
}\  \calI=t_1\dots t_p \cs_1\dots\cs_q,\ \calJ=t_{p+1}\dots t_a \cs_{q+1}\dots\cs_b.
$$
As $\Psi^j$ is symmetric in the derivative indices, we can freely change the order of indices; so $\Psi^j_{\calI\calJ}$ and $\Psi^j_{t_1\dots t_a\cs_1\dots \cs_b}$ are identified. We set 
$$
|\calI|=|t_1\dots t_p \cs_1\dots\cs_q|=p+q
$$
and define the conjugation by
$$
\ol\calI=\ct_1\dots\ct_p s_1\dots s_q.
$$
We also allow $\calI$ to be an empty list and then set $|\calI|=0$. When we are only interested in specific contractions, we omit irrelevant indices from the notation. For example, if there is a contraction between $\Psi^i$ and $\Psi^j$, we write
$$
\Psi^i_t\Psi^j_\ct \quad\text{or}\quad
\Psi^i_\ct\Psi^j_t.
$$
When we do not need to specify the type of indices, we also use upper case indices
$$
\Psi^i_T\Psi^j_{\ol T},
$$
where $T$ can be holomorphic $t$ or antiholomorphic $\ct$.  Note that this notation {\em  does not} mean that $T$ runs though $1,\dots, n,\ol 1,\dots,\ol n$.
In the case a pair of indices in a factor is contracted, we write
$$
\Psi^j_{t\ct}=\Delta\Psi^j
$$
and call such contraction a {\em trace}.  If there are $m$ pairs of indices that are contracted, we write
$$
\Psi^j_{t_1\dots t_m s_1\dots s_p\ct_1\dots\ct_m \ol u_1\dots \ol u_q}=\Delta^m\Psi^j_{s_1\dots s_p \ol u_1\dots \ol u_q}.
$$

With these notations, an invariant of $\Psi$ 
\beq
\label{Lsum}
L(\Psi)=\sum_{l\in \La} a_l \,{\contr}^l(\partial^{(A_1,B_1)}\psi^1\otimes
\dots\otimes
\partial^{(A_\sigma,B_\sigma)}\psi^\sigma),
\eeq
is written simply as
$$
L(\Psi)=\sum_{l\in \La} a_l \,{\contr}^l(\Psi^1\cdots
\Psi^\sigma),
$$
where  we also omit the symbol  $\otimes$, or more simply as
\beq
\label{Lsums}
L(\Psi)=\sum_{l\in \La} a_l \, C^l(\Psi).
\eeq
What will often be important are {\it sublinear combinations} of such linear combinations. In particular, for any subset $K\subset\La$, we let:
$$
[K\| L(\Psi)]:=\sum_{l\in K} a_l \, C^l(\Psi).
$$
If we then consider further subsets $F\subset K$,  etc., (some condition that define subset of $K$), we will denote those by $[F\| K\| L(\Psi)],$ etc.

\subsubsection{The local divergence formula}
We consider any linear combination of complete  contractions $L(\Psi)$ of the form \eqref{Lsums},
where each term in $L(\Psi)$ has a given weight $w$,  given degree $\sigma$.  The main assumption of the local divergence formula is that $L(\Psi)$  integrates to zero in dimension $n$ with an $n\ge\sigma-1$, i.e.,
$$
\int L(\Psi)=0\quad \text{for any }\Psi\in C^\infty(\bC^n,\bR^\sigma).
$$
The integration is done over $\bC^n=\bR^{2n}$ with respect to the standard volume form; this is omitted here and in the following.

Now given any $L(\Psi)$, we explicitly write it out with the indices for $\Psi^1$ and $\Psi'=(\Psi^2,\dots,\Psi^\sigma)$. Write the indices on $\Psi^1$ as $\Psi^1_{\calI}$,  $\calI$ is contracted against  $\ol\calI$ on $\Psi'$. Thus we may write $C^l(\Psi)=\contr^l(\Psi^1_{\calI}\cdots\Psi^\sigma)$ as
$$
\Psi^1_{\calI}\pcontr^l_{\ol\calI}(\Psi'),
$$
where ${\ol\calI}$ are free indices for the partial contraction $\pcontr^l(\Psi')$ and the indices in $\ol\calI$ are  contracted against $\calI$ in $\Psi^1_{\calI}$.  Let us write each term of $L(\Psi)$ in  this way
\beq
L(\Psi)=\sum_{l\in \Lambda} a_l \Psi^1_{\calI}\pcontr^l_{\ol\calI}(\Psi').
\eeq
We can perform integrations by parts of for the derivatives $\pa_\calI$ on $\psi^1$. Then we get
$$
\int L(\Psi)=
\int \sum_{l\in \Lambda}(-1)^{|\calI|} a_l 
\psi^1\,
\pa_{\calI}\pcontr^l_{\ol\calI}(\Psi').
$$
The above integral will vanish for {\it all} scalar-valued functions  $\psi^1$. This implies a {\it local} equation: 
\beq\label{localeq-example}
\sum_{l\in \Lambda} (-1)^{|\calI|}a_l\,  \pa_{\calI}\pcontr^l_{\ol\calI}(\Psi')
=0.
\eeq
Here the list $\calI$ depends on $l\in\Lambda$. Expanding the derivatives $\pa_{\calI}$ by using Leibnitz rule, we obtain a linear relation among complete contractions of the derivatives of $\Psi'$. This new local equation is denoted by $\Local_1L(\Psi)=0$. We call it the 1-local divergence formula. 

The same argument can be applied to any factor $\psi^k$, $k\in \{1,\dots,\sigma\}$.  We denote the resulting local equation by:
$$
\Local_kL(\Psi)=0.
$$
We note that, by Theorem \ref{second_main_thm} (see also the paragraph below it),  the above equation holds \emph{formally} since $n\ge\sigma-1$.

\subsubsection{Formal operations on integral equations.}\label{int-by-part-sec}
A second tool that we will often use in this paper will be to go from an integral equation of the form
\[
\int L(\Psi)=0
\] 
to a \emph{new} integral equation of the form
\[
\int  \wt L(\Psi)=0,
\]
where the linear combination $\wt{L}(\Psi)$ arises from $L(\Psi)$ via some  \emph{formal operation}. We give an  example to facilitate the understanding of the arguments in the next sections.

Let us express $L(\Psi)$ as in \eqref{Lsum}. In each term ${\rm contr}^l(\dots)$, $l\in \La$, let us also consider the  number of contractions between the factors $\Psi^1,\Psi^\sigma$;  let $N(l)$ be that number. We may schematically express $L(\Psi)$ as follows:
\beq
L(\Psi)=\sum_{l\in \La} a_l \contr^l
\big(\Psi_{\mc{J}\mc{I}}^1\cdots
\Psi^\sigma_{\ol{\mc{J}}\mc{Y}}\big),
\eeq
with the convention that indices $\mc{J}$ in  $\Psi^1$ contract against the indices in $\ol{\mc{J}}$ in $\Psi^\sigma$ (hence $|\calJ|=N(l)$); $\calI, \mc{Y}$ are sets of indices that are contracted but not against each other.

Now,  let $M:=\max_{l\in\La} N(l)$ and let  $\La_M\subset \La$ be the index set of terms with $N(l)=M$. Let us denote by 
\beq
\label{oldeqn}
\wt{L}(\Psi)=\sum_{l\in \La_M} a_l \,\contr^l
(\Psi^1_{\mc{I}}\cdots
\Psi^\sigma_{\mc{Y}})
\eeq
the new sum of terms that arises from the terms indexed in $\La_M$ in $L(\Psi)$ by \emph{erasing} the indices in $\mc{J},\ol{\mc{J}}$ in the factors 
$\Psi^1, \Psi^\sigma$, respectively. Thus, by construction, $\wt{L}(\Psi)$ has terms with weight $w-M$ and degree $\sigma$.

\begin{lemma}
Let $L(\Psi)$ be a local invariant of $\Psi$ of degree $\sigma$.  Assume that there exists an $n\ge\sigma-1$ such that $L(\Psi)$ integrates to zero in dimension $n$ and that no term in $L(\Psi)$ contains traces. Then $\wt{L}(\Psi)$ also integrates to zero in dimension $n$.
\end{lemma}
\begin{proof} Consider the local equation $\Local_1L(\Psi)=0$.  This equation can be expressed schematically as:
\[
\sum_{l\in \La} a_l (-1)^{|\mc{J}\mc{I}|}
\partial_{\mc{J}\mc{I}}\big[\pcontr^l\!\big
(\Psi^2\cdots 
\Psi^\sigma_{\ol{\mc{J}}\mc{Y}}\big)\big]=0.
\]
This equation holds for any $\Psi'\in C^\infty_0(\bC^n,\bR^{\sigma-1})$.

We consider the sublinear combination $[M\|\Local_1L(\Psi)]$ which contains precisely $M$ traces in the factor $\Psi^\sigma$. Since the terms in $L(\Psi)$ did \emph{not} contain traces, we  see that such terms in the above  equation can arise \emph{if and only if} all the  derivatives $\partial_{\mc{J}}$ in 
$$
\partial_{\mc{J}\mc{I}}\big[\pcontr^l\!\big
(\Psi^2\cdots 
\Psi^\sigma_{\ol{\mc{J}}\mc{Y}}\big)\big]
$$
are forced to hit the factor  $ \Psi^\sigma_{\ol{\mc{J}}\mc{Y}}$. Now, observe that
\beq
\label{local*}
[M\|\Local_1L(\Psi)]=0.
\eeq
This is true since $  \Local_1L(\Psi)=0$ holds \emph{formally} (see \S\ref{section-SMT}), \emph{and} since the number of traces in any given factor in any given term in $\Local_1L(\Psi)$ remains invariant under the formal operations allowed. Thus, the terms in $ [M\|\Local_1L(\Psi)]$ cannot be cancelled by applying formal operations to any \emph{other} term in $  \Local_1L(\Psi)$. By construction, equation   \eqref{local*}  can be expressed as:
\beq
\label{local*again}
\sum_{l\in \La^M} a_l (-1)^{M+|\mc{I}|}
\partial_{\mc{I}}\big[\pcontr^l\!\big
(\Psi^2\cdots 
\Psi^\sigma_{\calJ\ol{\mc{J}}\mc{Y}}\big)\big]=0.
\eeq

Now, the second step is to perform a formal \emph{erasing of indices} in this local equation, obtaining a \emph{new} local equation.  To do this, notice that \eqref{local*again} must  \emph{again} hold formally. It follows  that if we let 
$$
\pcontr^l\!\big
(\Psi^2\cdots 
\Psi^\sigma_{\mc{Y}}\big)
$$
to be the terms that arise from the above equation by just \emph{erasing} the $M$ traces in the factor $\Psi^\sigma\!$, then we have 
the equation: 
\beq
\label{testeq}
\sum_{l\in \La^M} a_l (-1)^{|\mc{I}|}
 \partial_{\mc{I}}\big[\pcontr^l\!\big
(\Psi^2\cdots 
\Psi^\sigma_{\mc{Y}}\big)\big]=0.
\eeq

Finally, we can show that $\wt{L}(\Psi)$  integrates to zero in dimension $n$ by the just considering the integral of the above over $\bC^n$:
$$\int 
\sum_{l\in \La^M} a_l (-1)^{|\mc{I}|}\Psi^1
\partial_{\mc{I}}\big[\pcontr^l\!\big
(\Psi^2\cdots 
\Psi^\sigma_{\mc{Y}}\big)\big]=0,
$$
and integrating by parts the derivatives in $\partial_{\mc{I}}$; each such integration by parts of $\pa_\calI$ forces  each derivative $\partial_A, A\in \mc{I}$  to hit the factor $\Psi^1$. Thus, repeating this $|\mc{I}|$ times, we get
$$
\int\sum_{l\in \La^M} a_l\Psi^1_{\calI}
\pcontr^l\!\big
(\Psi^2\cdots 
\Psi^\sigma_{\mc{Y}}\big)=0.
$$
The integrand is exactly $\wt{L}(\Psi)$. 
\end{proof}

Further down, this process is simply referred to as ``integrating the local equation \eqref{testeq} and integrating by parts again" or ``reversing the order of integrations by parts." 

\section{Proof of Proposition \ref{mainprop} (a): the case of positive order}
\label{dgt0}
The proof is done by a multiple induction.  To formulate the steps, we will slightly re-state our proposition in a more general form.  We first generalize the notion of acceptable complete/partial contractions. As in the previous subsection, we consider complete/partial contractions of the form
\beq
\label{contraction}
C(\Psi)=\pcontr(\Psi^1\cdots\Psi^\sigma),
\quad\text{where } 
\Psi^j=\pa^{(A_j,B_j)}\psi^j.
\eeq

\begin{definition}\label{calLdef}
Consider a list $\calL=(\alpha_1,\beta_1,\dots,\alpha_\sigma,\beta_\sigma) \in \{0,1,2\}^{2\sigma}$. We  call a partial  contraction $C(\Psi)$  in the form  \eqref{contraction} {\em $\calL$-acceptable} if $A_j\ge\alpha_j$ and $B_j\ge\beta_j$ for all $j$.  We call the list $\calL$ the {\em restriction list}. For an $\calL$-acceptable contraction $C(\Psi)$, we say that a factor  $\pa^{(A_j,B_j)}\psi^j$ is {\em minimal} if $(A_j,B_j)=(\alpha_j,\beta_j)$. The {\em order} $d$ of  $C(\Psi)$ is defined by
$$
d=\sum_{i=1}^\sigma(A_j+B_j-\alpha_j-\beta_j).
$$
\end{definition}

For an $\calL$-acceptable $C(\Psi)$,  the order $d$ is positive if and only if it has at least one non-minimal factor.

We can now state a generalized version of Proposition \ref{mainprop} (a) that we will be proving:

\begin{proposition}
\label{main.prop} Let $L(\Psi)$ be an $\calL$-acceptable scalar-valued invariant   of degree $\sigma$ and order $d>0$.  Assume that $L(\Psi)$ integrates to zero in dimension $n$ for some $n\ge\sigma-1$. Then there exist $\calL$-acceptable forms $T_a(\Psi), T_\ca(\Psi)$ such that
\beq
L(\Psi)=\pa_\ca T_a(\Psi)+\pa_a T_{\ca}(\Psi).
\eeq
Moreover, if the assumption is slightly strengthened, then the conclusion may be slightly strengthened as well. Assume further that $L(\Psi)$ contains no traces  in all terms. Then, for each fixed two numbers $i,j\in \{1,\dots,\sigma\}$ with $i\ne j$,
\begin{itemize}
\item[(i)] 
$T_a(\Psi), T_\ca(\Psi)$ can be chosen to contain no traces and the free indices $a$, $\ca$  not to belong to $\Psi^i$. 
\item[(ii)]
If in addition there are no contractions of the form $\Psi^i_\cs\Psi^j_s$ in all terms of $L(\Psi)$, then $T_a(\Psi)$, $T_\ca(\Psi)$ can in addition be chosen so that there are no contractions of the form $\Psi^i_\cs\Psi^j_s$ and so that the free indices $a, \ca$ do not belong to either $\Psi^i,\Psi^j$.
\end{itemize}
\end{proposition}

Recall  that a trace is a contraction within a factor; see \S\ref{Notation-subsec}.

\begin{definition}  
In the setting of (i) above, we will call the factor $\Psi^i$ the {\em special factor}, and in (ii), we call $\Psi^j$ the {\em second special factor}.  A contraction between these factors $\Psi^i_\cs\Psi^j_s$ is called a {\em special contraction}.
\end{definition}

In most instances below, the special factor will be $\Psi^1$ and the second special factor (whenever applicable) will be $\Psi^\sigma$\!.

\subsection{Three main steps in the proof} 
We will prove the Proposition  by an induction on the weight $w$ of $L(\Psi)$: Assuming that Proposition \ref{main.prop} is true for all weights $w'< w$, we will prove it is true for the weight $w$.

There are three key steps in this proof, which we highlight as separate Propositions. The common assumption here is: 
\beq\label{assn}
\begin{aligned}
&\text{$L(\Psi)$ integrates to zero in a dimension $n$ with $\ge\sigma-1$ and}\\
&\text{has order $d>0$.}
\end{aligned}
\eeq

\begin{proposition}
\label{prop.no.trace}
Assume \eqref{assn}. Then there exist $\calL$-acceptable forms $T_a(\Psi), T_\ca(\Psi)$ so that: 
\beq
L(\Psi)=\pa_\ca T_a(\Psi)+\pa_a T_{\ol{a}}(\Psi)+L^\sharp(\Psi),
\eeq
where $L^\sharp(\Psi)$ stands for a new linear combination of  $\calL$-acceptable complete contractions  which contain no traces.
\end{proposition}

Proposition \ref{prop.no.trace} will be proven in \S\ref{Pf.no.trace}. Observe that if we can show this, we are reduced to showing Proposition \ref{main.prop} under the additional assumption that there is no trace in any term in $L(\Psi)$.  Our next Proposition applies to that setting.

\begin{proposition}
\label{prop.no.special}
Assume \eqref{assn} and that there are no  traces in any term in $L(\Psi)$. Then there exist $\calL$-acceptable forms $T_a(\Psi), T_\ca(\Psi)$ so that
\beq
L(\Psi)=\pa_\ca T_a(\Psi)+\pa_a T_{\ol{a}}(\Psi)
+L^\flat(\Psi),
\eeq
the free indices $a,\ca$ do not belong to $\Psi^1$, and  $L^\flat(\Psi)$ stands for a new linear combination of $\calL$-acceptable complete contractions which contain no  traces and no special contractions $\Special$.
\end{proposition}

Proposition \ref{prop.no.special} will be proven in \S \ref{Pf.no.special}. Observe that if we can show this, we are reduced to showing Proposition \ref{main.prop} under the additional assumption that there are no  traces, {\it and} no special contractions, in any term in $L(\Psi)$. 

\begin{proposition}
\label{prop.final.step}
Assume \eqref{assn} and that there are no traces  and no special contractions $\Special$ in any term in $L(\Psi)$. Then there exist $\calL$-acceptable forms $T_a(\Psi), T_{\ol{a}}(\Psi)$ such that 
\beq
L(\Psi)=\pa_\ca T_a(\Psi)+\pa_a T_{\ol{a}}(\Psi).
\eeq
Moreover, $T_a(\Psi)$, $T_\ca(\Psi)$ can be chosen so that  any of their terms  satisfy the following conditions:
\begin{itemize}
\item
The free indices $a$ and $\ca$ do not belong to either $\Psi^1$ or $\Psi^\sigma;$
\item
There are no traces and no special contractions $\Special$.
\end{itemize}
\end{proposition}

Proposition \ref{prop.final.step} will be proven in \S\ref{Pf.final.step}. 

\subsection{Proof of Proposition \ref{prop.no.trace}}
\label{Pf.no.trace}
We first introduce a procedure to remove traces from non-minimal factors.   Note that the following lemma does not use the assumption on the integral.

\begin{lemma}\label{lem:internal-trace}
Let $C(\Psi)$ be a complete contraction with a non-minimal factor $\Psi^i$ with at least one trace. Then there are $\calL$-acceptable forms $T_a$ and $T_\ca$ such that
$$
C(\Psi)-\pa_\ca T_a(\Psi)-\pa_a T_\ca(\Psi)=L'(\Psi),
$$
where $L'(\Psi)$ is a linear combination of partial contractions such that $\Psi^i$ in each term is minimal unless it has no traces; moreover the numbers of traces of all the other factors do not change. 
\end{lemma}

\begin{proof}
For  simplicity of notation, we assume $i=1$. Then $C(\Psi)$ is of the form
$$
\contr\big(\Delta\pa^{(A_1-1,B_1-1)}\psi^1\cdot\pcontr(\Psi')\big).
$$
If $A_1>\alpha_1$, then replacing $\Delta$ by $\pa_\ca$, we set
$$
T_\ca=\pcontr\big(\pa_\ca\pa^{(A_1-1,B_1-1)}\psi^1\cdot\pcontr(\Psi')\big).
$$
Then $T_\ca$ is $\calL$-acceptable and we have
$$
C(\Psi)-\pa_a T_\ca=
-\contr\big(\pa_\ca\pa^{(A_1-1,B_1-1)}\psi^1\cdot\pa_a\pcontr(\Psi)\big).
$$
The right-hand side has fewer  holomorphic derivatives and fewer traces for $\Psi^1$; moreover, the numbers of traces for other factors is preserved. 
If $B_1>\beta_1$, we can apply the same procedure with $\pa_a$ replaced by $\pa_\ca$.  Repeating these procedure, we obtain the lemma.
\end{proof}

From this point onwards, we may assume with no loss of generality that in any term in $L(\Psi)$ with a trace in $\Psi^1$, the factor $\Psi^1$ must be minimal. We show that we can simplify $L(\Psi)$ iteratively. This is done in steps:
\medskip

{\em Step 1: Remove double trace for $\psi^1$.}
Let $[\Delta^2\psi^1\|L(\psi)]$ be the sublinear combination consisting of terms with a factor $\Delta^2\psi^1$, which exists only when $\alpha_1=\beta_1=2$. (If there is no such term, we skip this step.) Then
define $[\erase\psi^1\|\Delta^2\psi^1\|L(\psi)]$ to stand for the new linear combination that arises by erasing $\Delta^2\psi^1$ from  $[\Delta^2\psi^1\|L(\psi)]$. Note that this procedure does not change the order $d$ as $\Delta^2\psi^1$ is minimal. We claim 
\beq\label{int_reduced_eq}
\int [\erase\psi^1\|\Delta^2\psi^1\|L(\psi)]=0.
\eeq
If this equation holds, since $[\erase\psi^1\|\Delta^2\psi^1\|L(\psi)]$ has lower weight, we may invoke the inductive assumption of Proposition \ref{main.prop}, to derive that there exist $\calL$-admissible forms $\wt T_a$ and $\wt T_\ca$ such that 
$$
[\erase\psi^1\|\Delta^2\psi^1\|L(\psi)]=\pa_\ca\wt T_a+\pa_a\wt T_\ca.
$$
Then setting $T_a=\Delta^2\psi^1\cdot\wt T_a$ and $T_\ca=\Delta^2\psi^1\cdot\wt T_\ca$, we have
$$
[\Delta^2\psi^1\|L(\psi)]=\pa_\ca T_a+\pa_a T_\ca-
\partial_{\ca}\Delta^2\psi^1\cdot \wt T_a-\partial_a\Delta^2\psi^1\cdot 
\wt T_\ca.
$$
Thus, using Lemma \ref{lem:internal-trace}, we may remove a trace in $\partial_A\Delta^2\psi^1$ (and obtain $\partial_{a\cb}\Delta\psi^1$). 

To prove \eqref{int_reduced_eq}, we consider the local equation $\Local_2 L(\Psi)=0$. Pick out the  sublinear combination $[\Delta^2\psi^1\|\Local_2 L(\Psi)]$ in $\Local_2 L(\Psi)$ consisting of the terms that contain the  factor $\Delta^2\psi^1$. Such terms arise only when we move all derivatives on $\psi^2$ to factors $\Psi^j$, $j\ne 1$. Erasing $\Delta^2\psi^1$ from $[\Delta^2\psi^1\|\Local_2(\Psi)]$ and reversing the procedure of integration by part for $\psi^2\cdot[\erase\psi^1\|\Delta^2\psi^1\|\Local_2(\Psi)]$,  we obtain \eqref{int_reduced_eq}.

\medskip
\noindent
{\em Step 2: Remove a single trace from $\psi^1$.}
Now we assume that there is no term with a double trace $\Delta^2\psi^1$.  We next pick out the sublinear combination $[\Delta\psi^1\|L(\psi)]$ consisting of 
the terms with the factor $\pa_{a\cb}\Delta\psi^1$ (such a term exists only when $\alpha_1=\beta_1=2$; the case $\min\{\alpha_1,\beta_1\}=1$ will be discussed
later).  If there is no such term, one can skip this step. Erasing $\pa_{a\cb}\Delta\psi^1$ from $[\Delta\psi^1\|L(\psi)]$ we define  $[\erase\psi^1\|\Delta\psi^1\|L(\psi)]_{\ca b}$, where $\ca b$ which  were contracted against $\Psi^1$ are left as free indices. Then we have
\beq\label{Bad1from}
[\Delta\psi^1\|L(\psi)]=
\pa_{a\cb}\Delta\psi^1 \cdot[\erase\psi^1\|\Delta\psi^1\|L(\psi)]_{\ca b}.
\eeq
\begin{claim}\label{claim-int-eq-wtL}
The tensor-valued integral equation holds:
\beq\label{int-eq-wtL}
\int [\erase\psi^1\|\Delta\psi^1\|L(\psi)]_{\ca b}=0.
\eeq
\end{claim}

In fact, since we seek to apply the inductive assumption of Proposition \ref{main.prop}, we do not wish to work with a tensor-valued integral equation, but a scalar-valued variant. To derive this, we consider the integrand $ [\erase\psi^1\|\Delta\psi^1\|L(\psi)]_{\ca b}$ and break it up into sublinear combinations 
\[
\label{break-up}
[\erase\psi^1\|\Delta\psi^1\|L(\psi)]^k_{\ca b},\quad k=1,\dots, K,
\] 
depending on where the two free indices $\ca, b$ belong: Two terms in $ [\erase\psi^1\|\Delta\psi^1\|L(\psi)]_{\ca b}$ belong to the same sublinear combination $ [\erase\psi^1\|\Delta\psi^1\|L(\psi)]^k_{\ca b}$ if and only if the free indices $\ca, b$ belong to the same factors $\Psi^i,\Psi^j$.  We then define $[\erase\psi^1\|\Delta\psi^1\|L(\psi)]^k_{\erase \ca \!\erase{\, b}}$ to be the  \emph{complete} contraction that arises by formally \emph{erasing} the free indices $\pa_\ca, \pa_b$ from the factors $\Psi^i,\Psi^j$. Note that the complete contractions that we obtain are $\calL\rq{}$-acceptable, with respect to a \emph{new} restriction list:  if 
\[
\calL=(\alpha_1,\beta_1,\dots, \alpha_\sigma, \beta_\sigma)
\]
and the free indices $\ca,b$ belonged to factors $\Psi^i,\Psi^j$, then
\[
\calL\rq{}=(\alpha_2,\beta_2,\dots, \alpha_i,\beta'_i,\dots, \alpha'_j,\beta_j,\dots,\alpha_\sigma,\beta_\sigma),
\]
where $\beta\rq{}_i=\max\{\beta_i-1,0\}$, $\alpha'_i=\max\{\alpha_i-1,0\}$. We  {\em also} claim the scalar-valued integral equation:
\beq\label{int-eq-wtL-NEW}
\int [\erase\psi^1\|\Delta\psi^1\|L(\psi)]^k_{\erase \ca \!\erase{\, b}}=0.
\eeq

Since each $[\erase\psi^1\|\Delta\psi^1\|L(\psi)]^k_{\erase \ca \!\erase{\, b}}$ has  lower weight, by the inductive assumption of Proposition \ref{main.prop} to derive that there exist $\calL\rq{}$-admissible forms (for $\psi'=(\psi^2,\dots,\psi^\sigma)$) $\wt T^k(\psi')_{c}$  and $\wt T^k(\psi')_{\cc}$ so that
\beq
\label{eq1}
[\erase\psi^1\|\Delta\psi^1\|L(\psi)]_{\erase \ca \!\erase{\, b}}^k=\pa_\cc \wt T^k_{c}(\psi')
+\pa_c \wt T^k_{\cc}(\psi').
\eeq
Now, for each $k\in \{1,\dots, K\}$, we can add free indices $\pa_\ca, \pa_b$ to the factors $\Psi^i, \Psi^j$ from which they were erased. (Note that the tensor fields that are produced then are $\calL$-acceptable if we multiply them by $\pa_{a\cb}\Delta\psi_1$.) This produces a new true equation, of the form: 
\beq
\label{eq2}
[\erase\psi^1\|\Delta\psi^1\|L(\psi)]^k_{\ca b}=\pa_\cc \wt T^k_{c\ca b}(\psi')
+\pa_c \wt T^k_{\cc\,\ca b}(\psi').
\eeq
We now add  the above equations for $k\in \{1, \dots, K\}$ to derive: 
\beq\label{eq3}
[\erase\psi^1\|\Delta\psi^1\|L(\psi)]_{\ca b}=\pa_\cc \wt T_{c\ca b}(\psi')
+\pa_c \wt T_{\cc\,\ca b}(\psi').
\eeq

Then setting $T_c(\psi)=\pa_{a\cb}\Delta\psi^1\cdot \wt T_{cb\ca}(\psi')$ and $T_\cc(\psi)=\pa_{a\cb}\Delta\psi^1\cdot \wt T_{\cc b\ca}(\psi')$, we have
$$
\begin{aligned}[]
[\Delta\psi^1\|L(\psi)]=&\pa_\cc  T_{cb\ca}(\psi')
+\pa_c  T_{\cc b\ca}(\psi')\\
&-\pa_{ca\cb}\Delta\psi^1\cdot T_{\cc b\ca}(\psi')
-\pa_{\cc a\cb}\Delta\psi^1\cdot T_{cb\ca}(\psi').
\end{aligned}
$$
Since $\pa_{ca\cb}\Delta\psi^1$  and $\pa_{\cc a\cb}\Delta\psi^1$ are not minimal, we may apply Lemma \ref{lem:internal-trace} to the last line and remove the trace in $\Psi^1$ by subtracting a divergence.  

If $(\alpha_1, \beta_1)=(2,1)$, then we can apply the same argument for $\pa_a\Delta\psi^1$.  Similarly for the case $(\alpha_1, \beta_1)=(1,2)$. If $(\alpha_1, \beta_1)=(1,1)$, we can apply the argument of Step 1 to $\Delta\psi^1$ in place of $\Delta^2\psi^1$.

\begin{proof}[Proof of Claim \ref{claim-int-eq-wtL}]
To prove \eqref{int-eq-wtL} (and, in fact, the equations \eqref{int-eq-wtL-NEW}), we consider the local equation $\Local_2L(\Psi)=0$.  Pick out the terms in $\Local_2L(\Psi)$ with the following factors:
\beq\label{step2M3terms}
\pa_{a\cb}\Delta\psi^1,\quad
\pa_a\Delta^2\psi^1,\quad
\pa_\cb\Delta^2\psi^1,\quad
\Delta^3\psi^1.
\eeq
By construction, such factors can only arise from the terms in $[\Delta\psi^1\|L(\Psi)]$: terms containing the first factor arise only when 
all the derivatives on $\psi^2$ are forced to hit only factors $\Psi^j, j\ne 1$.\footnote{In other words, they are not allowed to hit the factor 
$\Psi^1$.} The latter three factors arise, respectively, by the integration by part of the pairs
$$
\pa_{a\cb}\Delta\psi^1 \Psi^2_b,\quad
\pa_{a\cb}\Delta\psi^1\Psi^2_\ca,\quad
\pa_{a\cb}\Delta\psi^1\Psi^2_{b\ca},
$$
where the derivatives on $\Psi^2$ omitted here are only allowed to hit factors $\Psi^j, j\ne 1$. To derive our new integral equations, we consider the local equation 
$$
[\Delta\psi^1\|\Local_2L(\Psi)]=0
$$
and pick out the sublinear combinations containing the factors in \eqref{step2M3terms}. These four sublinear combinations vanish separately, since $[\Delta\psi^1\|\Local_2L(\Psi)]$ vanishes formally as it has degree $\sigma-1\le n$. In these resulting four equations, we 
\emph{erase} the factors \eqref{step2M3terms} and make the indices $(\ca, b)$, $\ca$, $b$ free respectively for the first three. We denote the resulting expressions, respectively, by $I_{b\ca}(\Psi'')$, $I_{\ca}(\Psi'')$, $I_{b}(\Psi'')$, $I(\Psi'')$; these also vanish separately. We then consider the new local equation 
\beq
\label{new-loc}
\pa_{b\ca}\psi^2I(\Psi'')+\pa_\ca\psi^2 I_b(\Psi'')+\pa_a\psi^2 I_\cb(\Psi'')+
\psi^2_{a\cb} I(\Psi'')=0.
\eeq
Integrating this and reversing the process of integrations by parts for $I_{b\ca}, I_b, I_\cb, I$, we derive the integral equation \eqref{int-eq-wtL}. To derive  \eqref{int-eq-wtL-NEW} we just break up \eqref{new-loc} into sublinear combinations that vanish separately, according to how many free indices a term contains and which factors they belong to. These sublinear combinations must vanish separately. Integrating the resulting local equations yields  \eqref{int-eq-wtL-NEW}. 
\end{proof}

We make a technical remark here, which is helpful in the rest of this proof.

\begin{remark}
\label{KEY} 
The argument above proceeded by deriving the new system of integral equations \eqref{int-eq-wtL-NEW} and then applying the inductive assumption of Proposition \ref{main.prop} to each of those. The result is the system of local equations \eqref{eq2}. 

However, the system of integral equations \eqref{int-eq-wtL-NEW} also yields the tensor-valued integral equation \eqref{int-eq-wtL}, and the new local equations we derive in \eqref{eq2} is equivalent to the tensor-valued local equations \eqref{eq3}. On a formal level, we could then say that we apply the inductive assumption of Proposition \ref{main.prop} to the (tensor-valued) integral equation \eqref{int-eq-wtL} to derive \eqref{eq3}. But the proof of this proceeds via the \emph{system} of integral equations \eqref{int-eq-wtL-NEW} and the \emph{system} of local equations \eqref{eq2}. 

Nonetheless, since this argument is very general, we will adopt the language convention of invoking the inductive assumption of Proposition  \ref{main.prop} to a tensor-valued integral equation, to derive a tensor-valued local equation. While strictly speaking Proposition \ref{main.prop} is only applicable to scalar-valued integral equations,  in all instances where we invoke this below, it can be proven by way of thinking of the system of the tensor-valued equations involved as a system of scalar-valued equations. The strict proof of how this is done follows exactly the lines adopted here (and in fact is easier in general). We skip this technical point on the instances below for the reader's convenience, in order for the main ideas to not be obscured by this  technical point.
\end{remark}

{\em Step 3: Remove traces for $\psi^k\!,$ $k\ge2$.}
We may now assume with no loss of generality that there are no traces in $\Psi^1$ in any term in $L(\Psi)$. We next remove traces from all factors $\psi^k$, $k\ge2$. In view of Lemma \ref{lem:internal-trace}, we can assume that any factor with a trace is minimal. Such a factor is of the form
$$
\Delta^2\psi^k,\quad \pa_*\Delta\psi^k,
$$
where $*$ stands for indices of type $(\alpha_k-1,\beta_k-1)$.  We denote the numbers of these factors (when $k$ runs through $2,\dots,\sigma$) in a term respectively by $\gamma_1(l),\gamma_2(l)$, $l\in\La$. Let $M_1$ be the maximum of $\gamma_1$ among the terms in $L(\Psi)$ and denote by 
$$
[M_1\Delta^2\psi\|L(\Psi)]
$$
the sublinear combination with $\gamma_1=M_1$. (If $M_1=0$, we set $[M_1\Delta^2\psi\|L(\Psi)]=L(\Psi)$; we always use this convention in this subsection.)
Then let $M_2$ be the maximum of $\gamma_2$ among the term of $[M_1\Delta^2\psi\|L(\Psi)]$ and denote by 
$$
L^{M_1,M_2}(\Psi)=[M_2\Delta\psi\|M_1\Delta^2\psi\|L(\Psi)]
$$ 
the sublinear combination of $[M_1\Delta^2\psi\|L(\Psi)]$ with $\gamma_2=M_2$. We claim that we can choose $\calL$-acceptable forms $T_a$ and $T_\ca$ so that
\beq\label{LMM-claim}
L^{M_1,M_2}(\Psi)-\pa_\ca T_a-\pa_a T_\ca=L'(\Psi),
\eeq
where each term of $L'(\Psi)$ has smaller $(\gamma_1,\gamma_2)$ in  the lexicographical order, i.e., $\gamma_1<M_1$ or $\gamma_1=M_1$ but  $\gamma_2<M_2$. Clearly, if we can show this, then by iterative repetition we can derive our claim.

To prove the claim \eqref{LMM-claim},  we study the contractions between $\Psi^1$ and factors $\Delta\Psi^k$ with $k\ge 2$. There are three possible types of such contractions (irrelevant indices are omitted):
$$
\Psi_{a\cb *}^1\pa_{b\ca}\Delta\psi^k,
\quad
\Psi_{a*}^1\pa_{\ca*}\Delta\psi^k,
\quad
\Psi_{\ca *}^1\pa_{a*}\Delta\psi^k.
$$
here $*$ stands for the indices (possibly be empty) that do not give a contraction between $\Psi^1\Psi^k$. 
For each term of $L^{M_1,M_2}(\Psi)$, we count the numbers of these pairs; they are respectively denoted by $\gamma_3,\gamma_4,\gamma_5$.
Let $M_3$ be the maximum of $\gamma_3$ among the terms in $L^{M_1,M_2}(\Psi)$ and denote by
\beq\label{M3L}
[M_3\Psi^1_{a\cb}\pa_{b\ca}\Delta\psi\|L^{M_1,M_2}(\Psi)]
\eeq
the sublinear combination with $\gamma_3=M_3$.  Let $M_4$ be the maximum of $\gamma_4$ among the terms in \eqref{M3L}, and denote the sublinear combination with $\gamma_4=M_4$ by
\beq\label{M4L}
[M_4\Psi^1_{a}\pa_{\ca}\Delta\psi\|M_3\Psi^1_{a\cb}\pa_{b\ca}\Delta
\psi\|L^{M_1,M_2}(\Psi)].
\eeq
Finally, let $M_5$ be the maximum of $\gamma_5$ among the terms in \eqref{M4L}. Take the sublinear combination with $\gamma_5=M_5$ and set
\beq\label{M5L}
L^{M_1,M_2,\dots,M_5}(\Psi)=[M_5\Psi^1_{\ca}\pa_a\Delta\psi\|M_4\Psi^1_a\pa_{\ca}\Delta\psi\|M_3\Psi^1_{a\cb}\pa_{b\ca}\Delta\psi\|L^{M_1,M_2}(\Psi)].
\eeq

Now we pick the $m$ $(=M_1+M_2)$ factors with contractions in a term of \eqref{M5L}:
\beq\label{psilist}
\Delta^*\psi^{k_1},\Delta^*\psi^{k_2},\dots,\Delta^*\psi^{k_m}.
\eeq
This defines a subset $\{k_1,k_2,\dots,k_m\}\subset\{2,\dots,\sigma\}$. We split $L^{M_1,M_2,\dots,M_5}(\Psi)$ according to $\{k_1,k_2,\dots,k_m\}$: the sublinear combination with the factors \eqref{psilist} is denoted by
\beq\label{M6L}
L^{M_1,M_2,\dots,M_5}
_{k_1,k_2,\dots,k_m}(\Psi).
\eeq
From this expression, we erase the terms in \eqref{psilist}; the indices that contracted against these factors are left as free indices. The resulting expression is denoted by
\beq\label{eraseM5L}
[\Delta\erase\psi\|L^{M_1,M_2,\dots,M_5}
_{k_1,k_2,\dots,k_m}(\Psi)].
\eeq
Note that each term is a partial contraction of factors $\Psi''=(\Psi^{j_1},\dots,\Psi^{j_t})$,
where 
$$
\{j_1,j_2,\dots, j_t\}=\{1,2,\dots,\sigma\}\setminus\{k_1,k_2,\dots,k_m\}.
$$

\begin{claim} \label{lem.intM5L} The following integral equation holds:
\beq\label{intM5L}
\int
[\Delta\erase\psi\|L^{M_1,M_2,\dots,M_5}
_{k_1,k_2,\dots,k_m}(\Psi)]
=0.
\eeq
\end{claim} 

Before proving the claim, let us complete the proof of Proposition \ref{prop.no.trace} using \eqref{intM5L}. Since the integrand \eqref{intM5L} has lower weight and has no trace, by the inductive assumption of Proposition \ref{main.prop} (see Remark \ref{KEY}), we can find $\wt T_a(\Psi''), \wt T_\ca(\Psi'')$ without traces such 
that
$$
[\Delta\erase\psi\|L^{M_1,M_2,\dots,M_5}
_{k_1,k_2,\dots,k_m}(\Psi)]=
\pa_\ca \wt T_a(\Psi'')+\pa_a \wt T_\ca(\Psi'').
$$
We now put back the factors
$$
\Delta^*\psi^{k_1},\Delta^*\psi^{k_2},\dots,\Delta^*\psi^{k_m}
$$
into $[\Delta\erase\psi^*\|L^{M_1,M_2,\dots,M_5}_{k_1,k_2,\dots,k_m}(\Psi)]$. We also put these factors into $\wt T_a(\Psi'')$, $\wt T_\ca(\Psi'')$  and define $T_a(\Psi)$, $T_\ca(\Psi)$.  Then we obtain
$$
L^{M_1,M_2,\dots,M_5}
_{k_1,k_2,\dots,k_m}(\Psi)=
\pa_\ca T_a(\Psi)+\pa_a T_\ca(\Psi)+L'(\Psi),
$$
where $L'(\Psi)$ is a linear combination of terms with $M_1$ double traces and $M_2$ single traces and exactly one non-minimal factor with trace.  Then we can apply Lemma \ref{lem:internal-trace} to each term of $L'(\Psi)$ and reduce $M_1$ or $M_2$. This completes the inductive step.

\begin{proof}[Proof of Claim \ref{lem.intM5L}]
We consider the local equation $\Local_1L(\Psi)=0$. From the terms in $\Local_1L(\Psi)$, we pick up the factors in one of the forms:

\medskip

\noindent
(I) Minimal factors with traces $\alpha_k,\beta_k\in\{1,2\}$:
$$
\Delta^2\psi^k,\quad
\pa_{a\cb}\Delta\psi^k,
\quad
\pa_a\Delta\psi^k,
\quad
\pa_\ca\Delta\psi^k,
\quad
\Delta\psi^k.
$$
(II) Non-minimal factor of type $(\alpha_k,\beta_k+1)$:
$$
\begin{aligned}
&
\pa_\cb\Delta^2\psi^k\ \text{for }(\alpha_k,\beta_k)=(2,2),\\
&\Delta^2\psi^k\ \text{for }(\alpha_k,\beta_k)=(1,2).
\end{aligned}
$$
(III) Non-minimal factor of type $(\alpha_k+1,\beta_k)$: 
$$
\begin{aligned}
&
\pa_a\Delta^2\psi^k\ \text{for }(\alpha_k,\beta_k)=(2,2),\\
&\Delta^2\psi^k\ \text{for }(\alpha_k,\beta_k)=(2,1).
\end{aligned}
$$
(IV) Non-minimal factor of type $(\alpha_k+1,\beta_k+1)$:
$$
\Delta^3\psi^k\ \text{for }(\alpha_k,\beta_k)=(2,2).
$$
We count the number of these factors in each term $C(\Psi)$ of $\Local_1L(\Psi)$:

\medskip

$\alpha$: the number of  the factor $\Delta^2\psi^2$
in (I) in the term $C(\Psi)$;

$\beta$: the number of the factors 
$\pa_{a\cb}\Delta\psi^k,
\pa_a\Delta\psi^k,
\pa_\ca\Delta\psi^k,
\Delta\psi^k$ in (I);

$\gamma$: the numbers of the factors in (II) in the term $C(\Psi)$;

$\delta$: the numbers of the factors in (III) in the term $C(\Psi)$;

$\mu$: the numbers of the factors in (IV) in the term $C(\Psi)$.

\medskip

\noindent
We can tell how each factor above arises in the integration by parts. The factors in (I) are minimal and no derivative is applied in the integration by parts.  The factors in (II), (III), (IV) are respectively derived from the pairs
$$
\Psi^1_{a*}\pa_{\ca*}\Delta\psi^k,\quad
\Psi^1_{\ca*}\pa_{a*}\Delta\psi^k,\quad
\Psi^1_{b\ca*}\pa_{a\cb}\Delta\psi^k
$$
by moving the indices $\Psi^1$ to $\Psi^k$; here $*$ stands for the indices that do not give contraction between $\Psi^1\Psi^k$ and these derivatives on $\Psi^1$ are moved to factors $\Psi^j$, $j\ne k$. Note that these are the complete list of pairing of indices between $\Psi^1$ and minimal $\Psi^k$ with contraction.

By this observation we see that maximum of $\alpha$ among the terms in $\Local_1L(\Psi)$ is $M_1$. Let $[M_1\|\Local_1L(\Psi)]$ be the sublinear combination with $\alpha=M_1$.  Then consider the maximum of $\beta+\gamma+\delta+\mu$ among the terms in $[M_1\|\Local_1L(\Psi)]$.  It is $M_2$ and we denote the corresponding sublinear combination by
$$
[M_2\|M_1\|\Local_1L(\Psi)].
$$ 
We then take the maximum of $\mu$, which agrees with $M_3$, and denote the sublinear combination by 
$$
[M_3\|M_2\|M_1\|\Local_1L(\Psi)].
$$
We take the maximum of $\gamma$, which is $M_4$, and denote the sublinear combination by 
$$
[M_4\|M_3\|M_2\|M_1\|\Local_1L(\Psi)].
$$
Finally, we take the maximum of $\delta$, which is $M_5$, and denote the sublinear combination by 
$$
L^{M_1,M_2,\dots,M_5}_{1}(\Psi')=[M_5\|M_4\|M_3\|M_2\|M_1\|\Local_1L(\Psi)].
$$
This sublinear combination vanishes separately since ${\rm Local}_1L(\Psi)=0$  holds formally. Note that the numbers of the factors in the lists (I)--(IV) are
$$
\begin{aligned}
\alpha & =M_1,\quad 
\beta=M_2-M_3-M_4-M_5,
\\
\gamma & =M_4,\quad
\delta=M_5,\quad
\mu=M_3.
\end{aligned}
$$
Take a term of $L^{M_1,\dots,M_5}_{1}(\Psi')$ and we write the factors in (I)--(IV) as
\beq\label{Delta-psi-list}
\Delta^*\psi^{k_1},
\Delta^*\psi^{k_2},\dots,
\Delta^*\psi^{k_m}.
\eeq
Fix the list $\{k_1,k_2,\dots,k_m\}$ and pick up from $L^{M_1,\dots,M_5}_{1}(\Psi')$ the terms with factors of the form \eqref{Delta-psi-list}. The resulting sublinear combination is denoted by
$$
L^{M_1,\dots,M_5}_{1,k_1,\dots,k_m}(\Psi'),
$$
which also vanishes separately. Then erase the factors listed in \eqref{Delta-psi-list}; this gives a linear combination
$$
[\erase\psi^{k_1},\dots, \erase\psi^{k_m}\| L^{M_1,\dots,M_5}_{1,k_1,\dots,k_m}(\Psi')],
$$
which also vanishes. By reversing the procedure of integration by parts for 
$$
\psi^1[\erase\psi^{k_1},\dots, \erase\psi^{k_m}\| L^{M_1,\dots,M_5}_{1,k_1,\dots,k_m}(\Psi')]
$$
we obtain $[\Delta\erase\psi\|L^{M_1,M_2,\dots,M_5}_{k_1,k_2,\dots,k_m}(\Psi)]$ modulo divergence. (Recall by Remark \ref{KEY} we can equivalently 
derive a \emph{system} of scalar-valued integral equations.) Thus we get \eqref{intM5L}.\end{proof}

\subsection{Proof of Proposition \ref{prop.no.special}}
\label{Pf.no.special}

Throughout this section we start with $L(\Psi)$ satisfying the assumptions of Proposition \ref{prop.no.special}: 
\beq\label{assn3.5}
\begin{aligned}
&\bullet\ \text{$L(\Psi)$ integrates to zero in a dimension $n$ with $n\ge\sigma-1$;}\\
&\bullet\ \text{$L(\Psi)$  has order $d>0$;}\\
&\bullet\ \text{All terms in $L(\Psi)$ have no traces.}\\
\end{aligned}
\eeq
Let us denote by $M$ the maximum number of special contractions  $\Special$ among the terms $C^l(\Psi)$ in $L(\Psi)$.  Let us denote by $[M\Special\|L(\Psi)]$ the sublinear combination of $L(\Psi)$ with precisely $M$ special contractions. We then claim

\begin{lemma}
\label{decr.A}
There exist $\mc{L}$-acceptable forms $T_a(\Psi), T_{\ca}(\Psi)$ such that
\beq
\label{eq.decr.A}
[M\Special\|L(\Psi)]=\pa_\ca T_a(\Psi)+\pa_a T_\ca(\Psi)+L^{ M-}(\Psi),
\eeq
where $L^{M-}(\Psi)$  stands for a new linear combination of  $\mc{L}$-acceptable partial contractions which contain no traces and at most $(M-1)$ special contractions.
\end{lemma}

It is clear that if we can show this lemma,  then Proposition \ref{prop.no.special} will follow by iterative repetition.

We first prove a weak version of Lemma \ref{decr.A}.

\begin{lemma}\label{lem-non-minimal}
If each term in $[M\Special\|L(\Psi)]$ satisfies $A_\sigma>\alpha_\sigma$ and $B_1>\beta_1$ (see Definition \ref{calLdef}), then Lemma \ref{decr.A} holds.
\end{lemma}

\begin{proof}
Consider a term  $C(\Psi)$ in $[M\Special\|L(\Psi)]$. Consider any special contraction  $\Psi^1_\ca\Psi^\sigma_a$ in $C^l(\Psi)$.  We let $T_a(\Psi)$ be the $(1,0)$-form arising from $C(\Psi)$ by {erasing} the index $\ca$ from $\Psi^1$ and making $a$ into a free index.  We also let $T_\ca(\Psi)$ to be the $(0,1)$-form that arises from $-T_a(\Psi)$ by formally replacing the holomorphic free index $a$ by $\ca$. Note that by construction, the forms $T_a, T_\ca$ thus constructed have no traces, and the free index does not belong to $\Psi^1$. Moreover, they are $\mc{L}$-acceptable by the assumption $A_\sigma>\alpha_\sigma$ and $B_1>\beta_1$. We write 
\beq
C(\Psi)-\pa_\ca T_{a}-\pa_a T_\ca=\sum \wt{C}^l(\Psi).
\eeq 
Then the right hand side has less special contractions and also contains no traces. Applying the same procedure to each term of  $[M\Special\|L(\Psi)]$, we obtain the lemma.
\end{proof}

Note that the proof above does not use the integral equation for $L(\Psi)$.

In particular, if $M>\max\{\alpha_\sigma,\beta_1\}$, then we have  $A_\sigma>\alpha_\sigma$ and $B_1>\beta_1$ as $\min\{A_\sigma,B_1\}\ge M$; thus we can apply Lemma \ref{lem-non-minimal}. Hence we make the additional assumption for the rest of this proof that 
\beq\label{assn2}
M\le\max\{\alpha_\sigma, \beta_1\}.
\eeq

We denote the sublinear combination of $[M\Special\|L(\Psi)]$ consisting of the terms  with $A_\sigma=\alpha_\sigma$ by
$[\alpha_\sigma \|M\Special\|L(\Psi)].$

\begin{lemma}\label{lem.decr.A.Bad1}
Assume \eqref{assn3.5} and \eqref{assn2}. Then there exist $\mc{L}$-acceptable forms  $T_a(\Psi), T_{\ca}(\Psi)$ such that 
\beq
\label{eq.decr.A.Bad1}
[\alpha_\sigma\|M\Special\|L(\Psi)]=\pa_\ca T_a(\Psi)+\pa_a T_\ca(\Psi)+L^{
\alpha_\sigma+}(\Psi),
\eeq
where $L^{\alpha_\sigma+}(\Psi)$ is a linear combination of  $\mc{L}$-acceptable contractions  which contain no traces, have $M$ special contractions and also has $A_\sigma>\alpha_\sigma$. 
\end{lemma}

Using this lemma, by replacing $L(\Psi)$ by $L(\Psi)-\pa_\ca T_a(\Psi)-\pa_a T_\ca(\Psi)$, we may assume
$$
[\alpha_\sigma\|M\Special\|L(\Psi)]=0.
$$
In this setting, let $[\beta_1\|M\Special\|L(\Psi)]$ be the sublinear combination of terms in $[M\Special\|L(\Psi)]$ with $B_1=\beta_1$.  Then we claim

\begin{lemma}
\label{lem.decr.A.Bad2}
There exist $\mc{L}$-acceptable forms  $T_a(\Psi), T_{\ca}(\Psi)$ such that
\beq
\label{eq.decr.A.Bad2}
[\beta_1\|M\Special\|L(\Psi)]=\pa_\ca T_a(\Psi)+\pa_a T_\ca(\Psi)+
L^{\beta_1+}(\Psi),
\eeq
where $L^{\beta_1+}(\Psi)$ is a linear combination of $\mc{L}$-acceptable partial contractions which contain no traces, have $M$ special contractions, {\it and} also satisfy both $A_\sigma>\alpha_\sigma$ {\it and} $B_1 >\beta_1$.
\end{lemma}

Applying  Lemma \ref{lem-non-minimal} to the the remainder $L^{\beta_1+}(\Psi)$, we obtain  Lemma \ref{decr.A}. Thus it remains to prove Lemmas \ref{lem.decr.A.Bad1} and \ref{lem.decr.A.Bad2}. 

\subsubsection{Proof of Lemma \ref{lem.decr.A.Bad1}}
Consider all terms in $[\alpha_\sigma\|M\Special\|L(\Psi)]$. Let $b_\sigma$ be the minimum of $B_\sigma$, the number of anti-holomorphic indices in $\Psi^\sigma$.  We let  $[b_\sigma\|\alpha_\sigma\|M\Special\|L(\Psi)]$ be the sublinear combination of terms with $B_\sigma=b_\sigma$. We will then 
prove that there exist  $\mc{L}$-acceptable forms $T_a(\Psi)$ and $T_\ca(\Psi)$ such that
\beq 
\label{delta.step}
[b_\sigma\|\alpha_\sigma\|M\Special\|L(\Psi)]=\pa_\ca T_a(\Psi)+\pa_a T_\ca(\Psi)+L^{(\alpha_\sigma,b_\sigma)+}(\Psi),
\eeq
where $L^{(\alpha_\sigma,b_\sigma)+}(\Psi)$  stands for a new linear combination of $\mc{L}$-acceptable complete contractions  which contain no traces, {\it and} $M$ special contractions, and $(A_\sigma,B_\sigma)>(\alpha_\sigma,b_\sigma)$ in the lexicographical order, i.e., $A_\sigma>\alpha_\sigma$ or 
($A_\sigma=\alpha_\sigma$ and $B_\sigma>b_\sigma$).   Clearly, if we can prove \eqref{delta.step} then by iterative repetition, we can derive \eqref{eq.decr.A.Bad1}.

We will prove \eqref{delta.step} in steps: Let $N$ be the maximum number of contractions of the form $\Psi^1_a\Psi^\sigma_\ca$ 
among the terms in $ [b_\sigma\|\alpha_\sigma\|M\Special\|L(\Psi)]$.  Denote the sublinear combination with exactly $N$ such contractions by 
$$
L^{N,b_\sigma,\alpha_\sigma,M}(\Psi)=[N\Psi^1_a\Psi^\sigma_\ca\|b_\sigma\|\alpha_\sigma\|M\Special\|L(\Psi)].
$$
We will then show 
\beq 
\label{delta.step2}
L^{N,b_\sigma,\alpha_\sigma,M}(\Psi) = \pa_\ca T_a(\Psi)+\pa_a T_\ca(\Psi)
+L^{(\alpha_\sigma,b_\sigma)+}(\Psi)
\eeq
with exactly the same notational conventions as for \eqref{delta.step}. If we can show \eqref{delta.step2} then we can derive \eqref{delta.step} by iterative repetition. 

We will show \eqref{delta.step2} by deriving a new integral equation. Erasing the factor $\Psi^\sigma$ from $L^{N,b_\sigma,\alpha_\sigma,M}(\Psi)$ and making the indices contracted against $\Psi^\sigma$ into free indices and then erasing the free indices that belong to $\Psi_1$, we obtain a tensor of type $(b_\sigma-N,\alpha_\sigma-M)$, which we denote by
$$
[\erase\Psi^\sigma\|L^{N,b_\sigma,\alpha_\sigma,M}(\Psi)].
$$

\begin{claim}\label{lem.new.int}
The following integral equation holds:
\beq
\label{new.int}
\int 
[\erase\Psi^\sigma\|L^{N,b_\sigma,\alpha_\sigma,M}(\Psi)]=0.
\eeq
\end{claim}

Before proving the claim, we will use \eqref{new.int} to prove \eqref{delta.step2}. Observe that the integral equation \eqref{new.int} falls under the inductive assumption of Proposition \ref{main.prop}, with a new restriction list  $\mc{L}':=(\alpha'_1,\beta'_1,\alpha_2,\beta_2,\dots,\alpha_{\sigma-1},\beta_{\sigma-1})$,
where $\alpha'_1=\alpha_1-N$, $\beta'_1=\beta_1-M$. (See Remark \ref{KEY}.) By the construction, there are no traces in any of the terms in $[\erase\Psi^\sigma\|L^{N,b_\sigma,\alpha_\sigma,M}(\Psi)]$.  Also the weight is now strictly lower than that of $L(\Psi)$,  by construction.  We choose $\Psi^1$ to be the special factor; there will be no second special factor here. Invoking the inductive assumption of Proposition \ref{main.prop} (i), we derive that there exist
$\mc{L}'$-acceptable forms $\wt T_a(\Psi'), \wt T_{\ca}(\Psi')$ of types $(b_\sigma-N+1,\alpha_\sigma-M)$ and $(b_\sigma-N,\alpha_\sigma-M+1)$, with the free indices $a,\ca$ {\it not} belonging to $\Psi^1$  so that: 
\beq 
\label{delta.setp}
[\erase\Psi^\sigma\|L^{N,b_\sigma,\alpha_\sigma,M}(\Psi)]=
\pa_\ca\wt T_a(\Psi')+\pa_a\wt T_\ca(\Psi').
\eeq
Now, add  $(N,M)$ holomorphic/anti-holomorphic free indices onto $\psi_1$ and multiply the resulting equation by a factor
$\Psi^\sigma=\pa^{(\alpha_\sigma,b_\sigma)}\psi^\sigma$;
then contract all indices in $\Psi^\sigma$ against the anti-holomorphic/holomorphic indices in 
the terms in \eqref{delta.setp} and against the added $(N,M)$ derivatives on $\psi^1$. This gives tensors
$$
T_a(\Psi)=\wt T_{a}(\Psi')\Psi^\sigma
\quad\text{and}\quad
T_\ca(\Psi)=\wt T_{\ca}(\Psi')\Psi^\sigma.
$$
We thus derive:
\beq 
\label{delta.setp'}
\begin{aligned}
L^{N,b_\sigma,\alpha_\sigma,M}(\Psi)=
&
\pa_{a} T_{\ca}(\Psi)+\pa_{\ca} T_{a}(\Psi)\\
&+\wt T_a(\Psi')\pa_{\ca}\Psi^\sigma
+\wt T_\ca(\Psi')\pa_{a}\Psi^\sigma.
\end{aligned}
\eeq
Now, we just observe that by construction, the terms in 
$\wt T_a(\Psi')\pa_{\ca}\Psi^\sigma$ and $\wt T_\ca(\Psi')\pa_{a}\Psi^\sigma$
have an additional derivative on $\Psi^\sigma$;
thus they are allowed in $L^{(\alpha_\sigma,b_\sigma)+}(\Psi)$.

\begin{proof}[Proof of Claim \ref{lem.new.int}]
Consider the local equation
$$
\Local_1L(\Psi)=0
$$
and take the sublinear combination $[\Delta^{M+N}\Psi^\sigma\|\Local_1 L(\Psi)]$ of the left-hand side consisting of terms with  factor 
$\Psi^\sigma=\pa^{(\alpha_\sigma+N,b_\sigma+M)}\psi^\sigma$ and with $M+N$ traces in that factor. Clearly this sublinear combination vanishes separately:
\beq
\label{loc-sharp}
[\Delta^{M+N}\Psi^\sigma\|\Local_1L(\Psi)]=0.
\eeq
We claim that $[\Delta^{M+N}\Psi^\sigma\|\Local_1L(\Psi)]$ arises exclusively from $L^{N,b_\sigma,\alpha_\sigma,M}(\Psi)$ by making each contraction of the form $\Psi^1_A\Psi^\sigma_{\ol A}$  ($A$ holomorphic or anti-holomorphic) give $\Psi^\sigma_{A\ol A}$ after the integrations by parts  in $\Local_1L(\Psi)$; we also {\it do not} allow any of the other indices in  $\Psi^1$ to hit $\Psi^\sigma$. In other words, we claim that \eqref{loc-sharp} can also be described as 
follows: write out 
\beq
\label{write.out}
L^{N,b_\sigma,\alpha_\sigma,M}(\Psi)
= \sum_{u\in U}a_u C^u(\Psi),
\eeq
where each term in the right-hand side  can be written in the form 
\beq
\label{Cuform}
C^u(\Psi)=
\Psi^1_{\ol\calJ\,\ol\calK}C_\flat^u(\Psi'')_{\calJ\ol\calM}
\Psi^\sigma_{\calM\calK}.
\eeq
Here $\Psi''$ is shorthand for the factors $\Psi^2,\dots,\Psi^{\sigma-1}$ and $\calJ,\calK,\calM$ are lists of holomorphic and anti-holomorphic indices; 
$\ol\calJ,\ol\calK,\ol\calM$ are the lists of indices obtained by taking conjugate. We know that the list $\calM\calK$ has type $(\alpha_\sigma,b_\sigma)$,
$\calK$ has type $(M,N)$ and $\calM$ has type $(\alpha_\sigma-M,b_\sigma-N)$. Let us consider the partial contraction 
\beq
\label{CuCut}
C_\flat^u(\Psi'')_{\calJ\ol\calM}
\Psi^\sigma_{\calM\calK}.
\eeq  
that arises by formally erasing the factor $\Psi_1$ and making all indices that it contracted against into free indices. Then our claim is that:
\beq
\label{explanation}
{[}\Delta^{M+N}\Psi^\sigma\|\Local_1L(\Psi)]
=
\sum_{u\in U}a'_u\pa_{\ol \calJ} C_\flat^u(\Psi'')_{\calJ\,\ol\calM}
\cdot\Delta^{M+N}\Psi^\sigma_{\calM},
\eeq 
where $a'_u=(-1)^{|\calJ|} a_u$.

Let us prove \eqref{explanation}. Consider any term $C^*(\Psi')$ in 
${[}\Delta^{M+N}\Psi^\sigma\|\Local_1L(\Psi)]$,
which has arisen from a contraction $C^l(\Psi)$ in $L(\Psi)$. Since by hypothesis there exist no traces in  $C^l(\Psi)$,  all of the $M+N$ contractions in $\Psi^\sigma$ of $C^*(\Psi')$ must have arisen from contractions between $\Psi^1,\Psi^\sigma$ in $C^l(\Psi)$. Moreover, since by definition $\Psi^\sigma$ in $C^*(\Psi')$ has precisely $(\alpha_\sigma-M)$ holomorphic indices {\it not} involved in a trace, $C^l$ must have had {\it at least} $M$ contractions of the form $\Psi^1_\ca
\Psi^\sigma_a$. Recall that $M$ is the {\it maximum} number of such contractions for all terms $C^l(\Psi)$ in $L(\Psi)$; thus the number of such contractions is exactly $M$. Therefore $C^l(\Psi)$ is a term in $[\alpha_\sigma\|M\Special\|L(\Psi)]$. This also means that {\it all} the remaining $N$ contractions in $C^*(\Psi)$ must have arisen from contractions of the form $\Psi^1_a\Psi^\sigma_\ca$; and since there are $(b_\sigma-N)$ antiholomorphic indices in $\Psi^\sigma$ which are not involved in a trace, it follows that $C^l(\Psi)$ must have belonged to  $L^{N,b_\sigma,\alpha_\sigma,M}(\Psi)$. Finally, by the analysis above, we see that the terms in $[\Delta^{M+N}\Psi^\sigma\|\Local_1C^l(\Psi)]$ arise precisely by the procedure of making \eqref{CuCut} from \eqref{Cuform}. This proves \eqref{explanation}.  

Since the left-hand side of \eqref{explanation} vanishes formally, by erasing $\Delta^{M+N}\Psi^\sigma_\calM$ and multiplying by $\Psi^1_{\ol \calK}$, we have
$$
\sum_{u\in U} a'_u 
\Psi^1_{\ol \calK}\cdot\pa_\calJ C_\flat^u(\Psi'')_{\ol \calJ\,\ol\calM}=0.
$$
Now, integrate the equation and integrate by parts (as in \S\ref{int-by-part-sec}), which moves $\ol\calJ$  to $\psi^1$. This gives 
$$
\int
\sum_{u\in U} a_u 
\Psi^1_{\ol \calJ\,\ol \calK}\cdot C_\flat^u(\Psi'')_{\calJ\,\ol\calM}=0,
$$
which is precisely \eqref{new.int}.
\end{proof}

\subsubsection{Proof of Lemma \ref{lem.decr.A.Bad2} }
The proof is essentially identical to that of Lemma \ref{lem.decr.A.Bad1} with the roles of $\Psi^1,\Psi^\sigma$ interchanged. The roles of holo\-morphic/anti-holo\-morphic indices in $\Psi_\sigma$ now corresponding to anti-holomo\-rphic/holo\-morphic indices in $\Psi_1$. The one difference now is that since 
$$
[\beta_1\|M\Special \|L(\Psi)]=0,
$$
when we reach the new integral equation \eqref{new.int}, all terms are now $\mc{L}'$-acceptable, with $\mc{L}'$ now being
$$
\mc{L}'=(\alpha_2,\beta_2,\dots, \alpha_{\sigma-1},\beta_{\sigma-1},\alpha'_\sigma,\beta'_\sigma),
$$
with $\alpha'_{\sigma}=\alpha_\sigma-M+1$, $\beta'_\sigma=\beta_\sigma-N$ (the difference is in the $+1$ in the first equation). This ensures that when we move from equation \eqref{delta.setp} to equation \eqref{delta.setp'}, the last correction terms we obtain in the right-hand side satisfies $B_1>M$ \emph{and} $A_\sigma>M$. These are precisely the terms allowed in the right-hand side of \eqref{eq.decr.A.Bad2}.

\subsection{Proof of Proposition \ref{prop.final.step}}
\label{Pf.final.step}
Throughout this section we start with $L(\Psi)$ satisfying the assumptions of Proposition \ref{prop.final.step}. Namely,
\beq\label{assn3.6}
\begin{aligned}
&\bullet\ \text{$L(\Psi)$ integrates to zero in a dimension $n$ with $n\ge\sigma-1$;}\\
&\bullet\ \text{$L(\Psi)$  has order $d>0$;}\\
&\bullet\ \text{$L(\Psi)$ contains no traces  and no special contractions $\Special$.}\\
\end{aligned}
\eeq
Recall that the forms $T_a, T_{\ol{a}}$ that we seek to construct should have the following properties:
\beq\label{assnT}
\begin{aligned}
&\bullet\ \text{
The free index $a, \ca$ does not belong to either $\Psi^1$ or $\Psi^\sigma$;}\\
&\bullet\ \text{There are no  traces and no special contractions 
$\Special$.}
\end{aligned}
\eeq
\begin{definition}
\label{types}
Given any  partial contraction $C^l(\Psi)$ in the form \eqref{contraction}, and given any number $h\in\{2,\dots,\sigma-1\}$, we let $\Pi'_l(h)$ and $\Pi''_l(h)$ stand, respectively, for the numbers of contractions of the forms:
$$
\Psi^1_\ca\Psi^h_a \quad\text{and}\quad
\Psi^h_\ca\Psi^\sigma_a.
$$
We also set $\Pi_l(h)=\Pi'_l(h)+\Pi''_l(h)$.
\end{definition}

Given a linear combination $L(\Psi)$ of terms in the form \eqref{contraction}, we let $M\{L(\Psi) \}$ stand for the maximum number $\Pi_l(h)$ among all 
$l\in\Lambda$ and all $h\in \{2,\dots, \sigma-1\}$. We then define $k\{L(\Psi) \}$ to be the minimum number  $h\in \{2, \dots, \sigma-1\}$ for  which $\Pi_l(h)=M\{L(\Psi) \}$ for some $l\in \Lambda$. (Thus by definition $\Pi_{h'}(l)<M\{L(\Psi) \}$ for all $l\in \Lambda$ and all $h'<k\{L(\Psi) \}$). We let 
$$
M:=M\{L(\Psi)\}\quad\text{ and }\quad k:=k\{L(\Psi)\}
$$ 
for brevity. The main claim that we show in this section is then the following:
\begin{lemma}
\label{step.improve}
Let  $L(\Psi)$ be an $\calL$-acceptable invariant that satisfies \eqref{assn3.6}. Then there exist $\mc{L}$-acceptable forms $T_a, T_{\ol{a}}$ which satisfy \eqref{assnT}  so that
\beq
\label{step.eqn}
L(\Psi)=\pa_\ca T_a+\pa_a T_\ca+
L^\sharp(\Psi),
\eeq
where $L^\sharp(\Psi)$ stands for an $\calL$-acceptable form that satisfies \eqref{assn3.6} and satisfies one of the estimates: 
\begin{itemize}
\item $M\{L^\sharp(\Psi)\}< M$ 
\item $M\{L^\sharp(\Psi)\}= M\  \text{and}\ 
k\{L^\sharp(\Psi)\}>k.$
\end{itemize}
\end{lemma}

It follows straightforwardly that if we can prove the above  lemma, then by iterative repetition we can derive our Proposition \ref{prop.final.step} also.

In this lemma,  we only need to study the terms with $\Pi_l(k)=M$.  So we let 
$$
[M,k||L(\Psi)]
$$
to stand for the sublinear combination of $L(\Psi)$ for which $\Pi_l(k)=M$. The first step of the proof Lemma \ref{step.improve} is to  subtract divergences
satisfying  \eqref{assnT} from $[M,k\|L(\Psi)]$ and make the resulting terms {\it normalized}  in the following sense.

\begin{definition}
\label{normalized}
A term $C^l(\Psi)$ with $\Pi_k(l)=M$ is called {\em normalized} if 
$$
\Pi''_k(l)=1,
$$
i.e., $C^l(\Psi)$ has $M-1$ contractions of type $\Psi^1_\ca\Psi^k_a$  and one contraction of type $\Psi^k_\ca\Psi^\sigma_a$.
\end{definition}
We then claim the following: 

\begin{lemma}
\label{normalize}
Let  $L(\Psi)$ be  $\calL$-acceptable and satisfy \eqref{assn3.6}. Then there exist $\mc{L}$-acceptable forms $T_a, T_{\ol{a}}$ which satisfy \eqref{assnT} so that
\beq
\label{step.eqn'}
[M,k\|L(\Psi)]=\pa_\ca T_a+\pa_a T_\ca+
[M,k\|L^{\rm norm}(\Psi)]+L^\sharp(\Psi),
\eeq
where $L^\sharp(\Psi)$ stands for a linear combination of terms in the form \eqref{contraction}, has all the properties of the terms described after 
\eqref{step.eqn}, while $[M,k\|L^{\rm norm}(\Psi)]$ has all the properties of the terms in the left-hand side  of \eqref{step.eqn'}, but in addition is normalized, in the language of Definition \ref{normalized}.
\end{lemma}
If we can prove the above, we are reduced to showing Lemma \ref{step.improve} under the additional assumption that all the terms in the sublinear combination 
$[M,k\|L(\Psi)]$ are normalized (as in Definition \ref{normalized}). Under that additional assumption, we claim that our Lemma \ref{step.improve} is true:

\begin{lemma}
\label{easy.step}
Assume  \eqref{assn3.6} and in addition that all terms in $ [M,k\|L(\Psi)]$ are normalized. Then \eqref{step.eqn} holds. 
\end{lemma}

{\it Proof of Lemma \ref{easy.step}:} 
For reference purposes, we explicitly write out $ [M,k\|L(\Psi)]$ as a linear combination of partial contractions:
\beq
\label{sum}
[M,k\|L(\Psi)]:=\sum_{r\in R} a_r C^r(\Psi).
\eeq
We then let $\wt{C}^r(\Psi)$ to stand for the new partial contraction that arises from $C^r(\Psi)$ by erasing the $(M-1)$ contractions 
$\Psi^1_\ca\Psi^k_a$ and the one contraction of type $\Psi^k_\ca\Psi^\sigma_a$. (Notice that the resulting terms have weight $(w-M)$). We then claim a new integral equation:
\beq
\label{int.new}
\int \sum_{r\in R} a_r \wt{C}^r(\Psi)=0.
\eeq
The terms in the above integral equation  are $\mc{L}'$-acceptable with respect to the restriction list:
$$
\mc{L}'=(\alpha_1,\beta'_1,\alpha_2,\dots, \alpha'_k,\beta'_k,\dots,\alpha'_\sigma,\beta_\sigma),
$$
where $\beta'_1=\max\{ \beta_1-M+1,0\}$, $\alpha'_k=\max\{\alpha_k-M+1,0 \}$, $\beta'_k=\max\{\beta_k-1,0\}$, $\alpha'_\sigma:=\max\{\alpha_\sigma-1,0\}$. 
We postpone the proof of \eqref{int.new} for a moment, to see how it implies our Lemma: 

We can invoke the inductive assumption of Proposition \ref{main.prop} to the above, with the extra restriction being applicable with the factors 
$\Psi^1,\Psi^\sigma$ being the first/second special factors. 
We derive that there exist $\mc{L}'$-acceptable forms $T_a(\Psi), T_{\ol{a}}(\Psi)$ satisfying \eqref{assnT} so that:
\beq
\sum_{r\in R} a_r \wt{C}^r(\Psi)=\pa_\ca T_a(\Psi)+\pa_a T_\ca(\Psi).
\eeq
Now, we use the fact that the above equation holds formally to add the missing $(M-1)$ contractions of the form $\Psi^1_\ca\Psi^k_a$ and the one 
contraction of the form $\Psi^k_\ca\Psi^\sigma_a$. The result is a new true equation which is precisely \eqref{step.eqn}. 

Thus matters are reduced to deriving \eqref{int.new}. We prove this by applying the $k$-local divergence formula to $\int L(\Psi)=0$, deriving: $\Local_k L(\Psi)=0$. 

Now, pick out the sublinear combination in the above with exactly $M$ contractions of the form $\Special$.  The resulting expression is denoted by 
 $[M\Special\|\Local_k\!$ $L(\Psi)]$, which  vanishes separately:
$$
[M\Special\|\Local_kL(\Psi)]=0.
$$
\begin{claim}\label{claim1}
Each term of $[M\Special\|\Local_kL(\Psi)]$ arises exclusively from 
$ [M,k\|L(\Psi)]$ after the integrations by parts in $\Local_kL(\Psi)$ by the following  procedures:
\begin{enumerate}
\item The $(M-1)$ holomorphic indices in $\Psi^k$ that contract against $\Psi^1$ are forced to hit $\Psi^\sigma$;
\item The one antiholomorphic index in $\Psi^k$ that contracts against $\Psi^\sigma$
is   forced to hit $\Psi^1$;
\item All the remaining indexes in $\Psi^k$ are allowed to hit any factor.
\end{enumerate}
\end{claim}

We postpone the proof of Claim \ref{claim1} for  a minute to see how it implies \eqref{int.new}. 
Let $R^{M,k}\subset R$ to be the index set of terms belonging to  $[M,k\|L(\Psi)]$ and let $\wh C ^r(\Psi)$ stands for the complete contraction that arises from $C^r(\Psi)$ by formally replacing the expression:
$$
\Psi^1_{\ol{s}_1\dots\ol{s}_{M-1}\calI}\cdots \Psi^k_{s_1\dots s_k\ol{t}\calJ}\cdots \Psi^\sigma_{t\calK}
$$
by 
$$
\Psi^1_{\ol{s}_1\dots\ol{s}_{M-1}\ol{t}\calI}\cdots \Psi^k_{\calJ}\cdots \Psi^\sigma_{s_1\dots s_kt\calK}.
$$
Then Claim \ref{claim1} implies that:
$$
(0=)[M\Special\|\Local_kL(\Psi)]=\Local_k\Bigg(
\sum_{r\in R^{M,k}} a_r \wh C^r(\Psi)\Bigg).
$$
Thus, erasing the $M$ indices $\cs_1\dots\cs_{M-1}\ct$ from $\Psi^1$ and $s_1\dots s_k t$ from $\Psi^\sigma$ in the above local equation and then integrating over $\mathbb{C}^n$ and formally integrating by parts again yields precisely \eqref{int.new}.

So matters are reduced to showing Claim \ref{claim1}. Recall that by the assumption \eqref{assn3.6}, there are no special contractions $\Special$ in {\it any} term in $L(\Psi)$.
Thus for any term in $\Local_kL(\Psi)$ (arising from a term $C^r(\Psi)$) with $M$ such contractions, each such contraction must have arisen from a contraction in $C^r(\Psi)$ of either $\Psi^1_\ca\Psi^k_a$ or $\Psi^k_\ca\Psi^\sigma_a$. By definition of $M\{ L(\Psi)\}$ and $k\{L(\Psi)\}$ there can be at most $M$ such contractions, and the terms with exactly $M$ such contractions must belong to $ [M,k\|L(\Psi)]$. Moreover, by the assumption of our Lemma, there can only be exactly $M-1$ such contractions of type $\Psi^1_\ca\Psi^k_a$ and one such contraction of type $\Psi^k_\ca\Psi^\sigma_a$. Now Claim \ref{claim1} follows immediately. This concludes the proof of Lemma \ref{easy.step}. \qed

\subsubsection{Proof of Lemma \ref{normalize}. }
We show this Lemma by a new iteration:

For each $C^l(\Psi)$ in $[M,k\|L(\Psi)]$, let 
\[
\tau_l=|\Pi''_k(l)-1|+1.
\]
Hence $C^l(\Psi)$ is normalized if and only if $\tau_l=1$. We claim that if we let 
$$
\tau=\max\{\tau_l:l\in \Lambda^{M,k}\}
$$
and we let $[\tau\Specialks\|M,k\|L(\Psi)]$ be the corresponding sublinear combination, then there exist one forms $T_a(\Psi)$ and $T_\ca(\Psi)$ satisfying \eqref{assnT} so that:
\beq
\label{iteration}
\begin{aligned}[]
[\tau\Specialks\|M,k\|L(\Psi)]=&\pa_\ca T_a(\Psi)+\pa_aT_\ca(\Psi)\\
&+{[M,k\|L(\Psi)]}^\sharp+
L^{\rm new}(\Psi),
\end{aligned}
\eeq  
where $L^{\rm new}(\Psi)$  stands for a generic linear combination as in Lemma \ref{normalize}, while 
$$
{[M,k\|L(\Psi)]}^\sharp=\sum_{l\in \wt{\Lambda}}a_lC^l(\Psi)
$$ 
stands for a linear combination of terms with all the properties of the terms in $[M,k\|L(\Psi)]$ {\it and in addition} $\tau_l<\tau$ for all $l\in \wt{\Lambda}$. If we can show this, then  iterating this step we derive Lemma \ref{normalize} and thus Proposition \ref{prop.final.step} also. So the rest of this section is devoted to showing \eqref{iteration}.
\medskip

We have two explicit constructions of divergences at our disposal which we will apply whenever the $(0,1)$ and $(1, 0)$-forms that we construct are 
$\mc{L}$-acceptable; these are similar in spirit to the constructions in Lemma \ref{lem-non-minimal}. 

Consider any term $C^l(\Psi)$ in $ [M,k\|L(\Psi)]$.  If $\Pi''_k(l)=1$, the term $C^l(\Psi)$ is normalized.  Otherwise, there are two cases we have to deal with:
$\Pi''_k(l)=0$ and  $\Pi''_k(l)>1$.

\medskip
\noindent
{\bf Procedure I.} Let $\Pi''_k(l)=0$ and further assume that
\beq\label{nonminimal1}
B_1>\beta_1\quad\text{ and}\quad 
A_k>\alpha_k.
\eeq

Let $T_a(\Psi)$ be the $(1,0)$-partial contraction obtained by erasing $\ca$ on $\Psi^1$ from $C^l(\Psi)$. Let also $T_\ca(\Psi)$ be the $(0,1)$-partial contraction that arises from $-T_a(\Psi)$ by changing the holomorphic index $a$ into an anti-holomorphic index $\ca$. $T_a$ and $T_\ca$ are $\calL$-acceptable by the assumption
\eqref{nonminimal1}.
Note that by construction, the forms $T_a$, $T_\ca$ thus constructed have no traces, no special contractions, and the free index does not belong to $\Psi_1$ or $\Psi_\sigma$. We also observe that we can write:
$$
C^l(\Psi)=\pa_\ca T_a(\Psi)+\pa_a T_\ca(\Psi)+\wt C^l(\Psi)+L^{\sharp,l}(\Psi),
$$
where $\wt C^l(\Psi)$ is a complete contraction that arises from $C^l(\Psi)$ by formally replacing the expression:
$$
\Psi^1_{\ca \calI}\cdots \Psi^k_{a\calJ}\cdots\Psi^\sigma_{\calK}
$$
by a new expression:
$$
\Psi^1_{\calI}\cdots\Psi^k_{\ca \calJ}\cdots\Psi^\sigma_{a\calK},
$$
which is normalized. Also $L^{\sharp,l}(\Psi)$ stands for a linear combination of terms as in $L^\sharp(\Psi)$ 
of \eqref{step.eqn'}.

\medskip
\noindent
{\bf Procedure II.} Let $\Pi''_k(l)>1$ and further assume that
\beq\label{nonminimal2}
B_k>\beta_k\quad\text{ and}\quad 
A_\sigma>\alpha_\sigma.
\eeq

In the term $C^l(\Psi)$, pick a contraction of type $\Psi^k_\ca\Psi^\sigma_a$ (there are at least two such contractions). Let $T_a(\Psi)$ be the $(1,0)$-partial 
contraction obtained from $-C^l(\Psi)$ by easing $a$ on $\Psi^\sigma$. Let also $T_\ca(\Psi)$ be the $(0,1)$-partial contraction that arises from $-T_a(\Psi)$ by changing the holomorphic index $a$ into an anti-holomorphic index $\ca$. By the assumption \eqref{nonminimal2}, $T_a$ and $T_\ca$ are $\calL$-acceptable.  Note that by construction, the forms $T_s$, $T_\ca$ thus constructed have no traces, no special contractions, and the free index does not belong to $\Psi^1$ or $\Psi^\sigma$.  We also observe that we can write: 
$$
C^l(\Psi)=\pa_\ca T_a(\Psi)+\pa_a T_\ca(\Psi)+\wt C^l(\Psi)+L_l^{\sharp}(\Psi),
$$
where $\wt C^l(\Psi)$ is a complete contraction that arises from $C^l(\Psi)$ by formally replacing the expression:
$$
\Psi^1_{A\dots B}\cdots\Psi^k_{\cs C\dots D}\cdots\Psi^\sigma_{sE\dots F}
$$
by a new expression:
$$
\Psi^1_{\cs A\dots B}\cdots\Psi^k_{sC\dots D}\cdots\Psi^\sigma_{E\dots F},
$$
which has fewer contractions of type $\Psi^k_\ca\Psi^\sigma_a$. Also $L_l^{\sharp}(\Psi)$ stands for a linear combination of terms as in $L^\sharp(\Psi)$ of \eqref{step.eqn'}. We can repeat this procedure until $\wt C^l(\Psi)$ becomes normalized as long as \eqref{nonminimal2} holds.

Using Procedures I and II, we can show \eqref{iteration} except if $[\tau\Specialks\|M,k\|L(\Psi)]$ contains terms with one of the following two properties:
\begin{itemize}
\item[(I)] 
$B_1=\beta_1$ or  $A_k=\alpha_k$, and  moreover $\Pi''_k(l)= 0$;
\item[(II)] 
$B_k=\beta_k$ or  $A_\sigma=\alpha_\sigma$, and moreover $\tau\ge 2$.
\end{itemize}
We are thus reduced to the setting where there are \emph{only} such terms in $[\tau\Specialks\|M,$ $k\|L(\Psi)]$. 

We consider the case where there exist terms as in (II) above. (The case of contractions as in (I) follows by an entirely analogous argument, which we skip.) We denote by  
\[
[\alpha_\sigma\| \tau\Specialks\|M,k\|L(\Psi)]
\]
the sublinear combination consisting of terms in $[\tau\Specialks\|M,k\|L(\Psi)]$ with $A_\sigma=\alpha_\sigma$.  We claim

\begin{lemma}
\label{lem.decr.A.Bad1'}
There exist $\mc{L}$-acceptable forms $T_a(\Psi), T_\ca(\Psi)$ satisfying \eqref{assnT} such that 
\beq
\label{eq.decr.A.Bad1'}
[\alpha_\sigma\| \tau\Specialks\|M,k\|L(\Psi)]
=\pa_\ca T_a(\Psi)+\pa_a T_\ca(\Psi)+L^{
\alpha_\sigma+}(\Psi),
\eeq
where $L^{\alpha_\sigma+}(\Psi)$ is a linear combination of  $\mc{L}$-acceptable complete contractions which contain no traces, no contractions
of the form $\Special$,   $\tau$  contractions of the form $\Specialks$ and $A_\sigma>\alpha_\sigma$. 
\end{lemma}

Using this lemma, we may assume
$$
[\alpha_\sigma\| \tau\Specialks\|M,k\|L(\Psi)]=0.
$$
In this setting, let $[\beta_k\| \tau\Specialks\|M,k\| L(\Psi)]$ be the sublinear combination of terms in $[\tau\Specialks\|M,k\| L(\Psi)]$ with $B_k=\beta_k$. 
Then we claim

\begin{lemma}
\label{lem.decr.A.Bad2'}
There exist
$\mc{L}$-acceptable forms $T_a(\Psi)$, $T_\ca(\Psi)$  satisfying \eqref{assnT} such that
\beq
\label{eq.decr.A.Bad2'}
[\beta_k\| \tau\Specialks\|M,k\|L(\Psi)]
=\pa_\ca T_a(\Psi)+\pa_a T_\ca(\Psi)+
L^{\beta_k+}(\Psi),
\eeq
where $L^{\beta_k+}(\Psi)$ is a linear combination of  $\mc{L}$-acceptable complete contractions which contain no traces, no contractions of the form $\Special$,   $\tau$  contractions of the form $\Specialks$ {\it and} $A_\sigma>\alpha_\sigma$ {\it and} $B_k >\beta_k$.
\end{lemma}

The proof of these Lemmas follows closely the proofs of Lemmas \ref{lem.decr.A.Bad1}
and \ref{lem.decr.A.Bad2}. In fact, half the proof is so similar that we only highlight the spots where the argument is altered. For the remaining half we need a modification of the earlier argument, which we explain in detail. 

\medskip

Consider all terms in $[\alpha_\sigma\| \tau\Specialks\|M,k\|L(\Psi)]$. Let $b_\sigma$ be the minimum of $B_\sigma$, the number of anti-holomorphic indices in $\Psi^\sigma$. We let $[b_\sigma\|\alpha_\sigma\| \tau\Specialks\|M, k\|L(\Psi)]$ be the sublinear combination of terms with $B_\sigma=b_\sigma$. We will then prove that there exist an $\mc{L}$-acceptable form $T_a(\Psi)$, $T_\ca(\Psi)$ satisfying  \eqref{assnT} such that
\beq 
\label{delta.step'}
[b_\sigma\|\alpha_\sigma\| \tau\Specialks\|M,k\|L(\Psi)] =\pa_\ca T_a(\Psi)+\pa_a T_\ca(\Psi)+L^{(\alpha_\sigma,b_\sigma)+}(\Psi),
\eeq
where $L^{(\alpha_\sigma,b_\sigma)+}(\Psi)$  stands for a new linear combination of $\mc{L}$-acceptable  contractions which contain no traces, no contractions $\Special$, $\tau$  contractions $\Specialks$, and $(A_\sigma,B_\sigma)>(\alpha_\sigma,b_\sigma)$ in the lexicographical order, i.e.,
$A_\sigma>\alpha_\sigma$ or ($A_\sigma=\alpha_\sigma$ and $B_\sigma>b_\sigma$).  Clearly, if we can prove \eqref{delta.step'} then by iterative repetition, we can derive \eqref{eq.decr.A.Bad1'}.

We will prove \eqref{delta.step'} in steps: Let $N$ be the maximum number of contractions of the form $\Psi^k_a\Psi^\sigma_\ca$ among the terms in $ [b_\sigma\|\alpha_\sigma\|\tau\Specialks\|M,k\|L(\Psi)]$. Denote the sublinear combination with exactly $N$ such contractions by 
$$
L^{N,b_\sigma,\alpha_\sigma,\tau}(\Psi)=[N\Psi^k_a\Psi^\sigma_{\ca}\|
b_\sigma\|\alpha_\sigma\| \tau\Specialks\|M,k\|L(\Psi)].
$$
We will then show 
\beq 
\label{delta.step2'}
L^{N,b_\sigma,\alpha_\sigma,\tau}(\Psi) = \pa_\ca T_a(\Psi)+\pa_a T_\ca(\Psi)
+L^{(\alpha_\sigma,b_\sigma)+}(\Psi)
\eeq
with exactly the same notational conventions as for \eqref{delta.step'}. If we can show \eqref{delta.step2'} then we can derive \eqref{delta.step'} by iterative repetition. 

We will show \eqref{delta.step2'} by deriving a new integral equation. We erase the factor $\Psi^\sigma$, which has type $(\alpha_\sigma,b_\sigma)$, from $L^{N,b_\sigma,\alpha_\sigma,\tau}(\Psi)$ and make the indices that contracted against $\Psi^\sigma$ into free indices; we also erase the $(M-\tau)$ 
contractions of the form $\Psi^1_\ca\Psi^k_a$. We thus obtain a tensor of type $(b_\sigma,\alpha_\sigma)$, which we denote by
$$
[\erase\Psi^\sigma\|L^{N,b_\sigma,\alpha_\sigma,\tau}(\Psi)].
$$

\begin{claim}\label{lem.new.int'}
The following integral equation holds:
\beq
\label{new.int'}
\int 
[\erase\Psi^\sigma\|L^{N,b_\sigma,\alpha_\sigma,\tau}(\Psi)]=0.
\eeq
\end{claim}

Before proving the claim, we will use \eqref{new.int'} to prove \eqref{delta.step2'}. Observe that the integral equation \eqref{new.int'} falls under the inductive assumption of Proposition \ref{main.prop} (see Remark \ref{KEY}), with a new restriction list  
$\mc{L}'=(\alpha_1,\beta'_1, \alpha_2,\beta_2,\dots,\alpha_{\sigma-1},\beta_{\sigma-1})$, $ \beta'_1=\beta_1-(M-\tau)$. 
By the construction, there are no traces in any of the terms in $[\erase\Psi^\sigma\|L^{N,b_\sigma,\alpha_\sigma,\tau}(\Psi)]$. Also the weight is now strictly lower than that of $L(\Psi)$, by construction. Furthermore in this case, we make $\Psi^1$ the special factor and $\Psi^k$ the second special factor. Observe that by construction there are no contractions of the form $\Psi^1_{\ol{s}}\Psi^k_s$.  So the requirement of Proposition \ref{main.prop} (ii) also holds.

Invoking the inductive assumption of Proposition \ref{main.prop}, we derive that   there exist $\mc{L}'$-acceptable forms $\wt T_a(\Psi'), \wt T_{\ca}(\Psi')$ of types $(b_\sigma+1,\alpha_\sigma)$ and $(b_\sigma,\alpha_\sigma+1)$, with the free indices $a,\ca$ {\it not} belonging to $\Psi^1$ nor to $\Psi^k$ so that: 
\beq 
\label{delta.setpn}
[\erase\Psi^\sigma\|L^{N,b_\sigma,\alpha_\sigma,\tau}(\Psi)]=
\pa_{\ca} \wt T_{a}(\Psi')+\pa_{a} \wt T_{\ca}(\Psi').
\eeq
Now, multiply the above by a factor $\Psi^\sigma=\pa^{(\alpha_\sigma,b_\sigma)}\psi^\sigma$, and contract all indices in $\Psi^\sigma$ against the last type $(b_\sigma, \alpha_\sigma)$ indices in the terms in \eqref{delta.setpn}.  We also add  the $(M-\tau)$ contractions of the form $\Psi^1_\cs\Psi^k_s$. This gives one forms
$$
T_a(\Psi)=\wt T_{a}(\Psi')\Psi^\sigma
\quad\text{ and }\quad
T_\ca(\Psi)=\wt T_{\ca}(\Psi')\Psi^\sigma.
$$
We thus derive:
\beq 
\label{delta.setp'n}
\begin{aligned}
L^{N,b_\sigma,\alpha_\sigma,\tau}(\Psi)=&
\pa_{\ca} T_{a}(\Psi)+\wt T_a(\Psi')\pa_{\ca}\Psi^\sigma\\
&+\pa_{a} T_{\ca}(\Psi)+\wt T_\ca(\Psi')\pa_{a}\Psi^\sigma.
\end{aligned}
\eeq
Now, we just observe that by construction, the terms in $\wt T_a(\Psi')\pa_{\ca}\Psi^\sigma$ and  $\wt T_\ca(\Psi')\pa_{a}\Psi^\sigma$ have an additional derivative on $\Psi^\sigma$; thus they are allowed in $L^{(\alpha_\sigma,b_\sigma)+}(\Psi)$ in \eqref{delta.step'}.

\begin{proof}[Proof of Claim \ref{lem.new.int'}]
Consider the local equation
\beq
\label{localk}
\Local_kL(\Psi)=0
\eeq
and consider the sublinear combination $[M \Special\|\Delta^N\psi^\sigma\|\Local_k L(\Psi)]$ of the left-hand side consisting of terms with  factor 
$\Psi^\sigma=\pa^{(\alpha_\sigma+N+M-\tau,b_\sigma)}\psi^\sigma$ and with $N$ traces in that factor, i.e., $\pa^{(\alpha_\sigma+M-\tau,b_\sigma-N)}\Delta^N\psi^\sigma$, \emph{and} with $M$ special contractions $\Special$. 

Clearly since \eqref{localk} holds formally, this sublinear combination vanishes separately (and formally again):
\beq
\label{loc-sharp'}
[M\Special \|\Delta^N\psi^\sigma\|\Local_k
L(\Psi)]=0.
\eeq
\begin{claim}
\label{last-claim}
The left-hand side of \eqref{loc-sharp'} arises exclusively from $L^{N,b_\sigma,\alpha_\sigma,\tau}(\Psi)$ after the integrations by parts $\Local_kL(\Psi)$ by the following procedures:
\begin{enumerate}
\item 
The $N$ contractions $\Psi^k_a\Psi^\sigma_\ca$  give $N$ traces $\Delta^N\Psi^\sigma$;

\item
The $\tau$ contractions  $\Specialks$ give $\tau$ contractions  $\Special$;

\item
The $(M-\tau)$ contractions  $\Specialonek$ give $(M-\tau)$ contractions  $\Special$;

\item
Any other indices in $\Psi^k$ are not allowed to hit $\Psi^\sigma$.
\end{enumerate}
\end{claim} 

We can check that Claim \ref{last-claim} implies Claim \ref{lem.new.int'} by repeating the argument after Claim   \ref{lem.new.int}:
Erase $\Psi^\sigma$, multiply by $\psi_k$, integrate over $\mathbb{C}^n$ and integrate by parts (reverse the procedure (4) above). 
So the rest of this proof is devoted to proving Claim \ref{last-claim}. 

Consider any term $C^*(\Psi')$ in $[M \Special\|\Delta^N\psi^\sigma\|\Local_k L(\Psi)]$, which has arisen from a contraction $C^l(\Psi)$ in $L(\Psi)$.  Since there were no special contractions $\Special$ in $C^l(\Psi)$, all such contractions in $C^*(\Psi')$ must have arisen from contractions in one of the two forms $\Specialonek$ or  $\Specialks$. Thus, since  $M$ is the maximum of $\Pi_k(l)$ for all $l$, it can only have arisen from terms with $\Pi_k(l)=M$, i.e., $C^l(\Psi)$ is a term in $[M,k\|L(\Psi)]$. Moreover all of the $M$ contractions in one of these two forms must give rise to a special contraction $\Special$. Observe that the holomorphic derivatives $\pa^{(\alpha_\sigma+M-\tau,b_\sigma-N)}\Delta^N\psi^\sigma$ in  $C^*(\Psi')$ include the ones that arose from the integration by parts of the  $(M-\tau_l)$ contractions $\Psi^1_\cs\Psi^k_s$ and the $N$ contractions  $\Psi^k_a\Psi^\sigma_\ca$. Hence the factor $\Psi^\sigma$ in $C^l(\Psi)$ has now been hit by at least  an additional $(M-\tau_l+N)$ holomorphic derivatives. By comparing the numbers of holomorphic indices in $\Psi^\sigma$ of $C^*(\Psi')$, we have
$$
\alpha_\sigma+M-\tau+N\ge A_\sigma+M-\tau_l+N.
$$
Thus combining with  $A_\sigma\ge\alpha_\sigma$, we get $\tau\le \tau_l$.   Since $\tau$ is the maximum of $\tau_l$, we conclude that $\tau=\tau_l$ and that $\Psi^\sigma$ in $C^l(\Psi)$ has exactly $\alpha_l$ holomorphic indices.  Therefore, $C^l(\Psi)$ is a term in
\beq\label{alphaPsi}
[\alpha_\sigma\|\tau\Specialks\|M,k\|L(\Psi)].
\eeq
We next consider the anti-holomorphic indices of $\Psi^\sigma$.  A term $C^l(\Psi)$ that contributes to $[M\Special\|\Delta^N\psi^\sigma\|\Local_k
L(\Psi)]$ must have  $B_\sigma\le b_\sigma$ by construction. But since $b_\sigma$ is the minimum value of $B_\sigma$ among terms in \eqref{alphaPsi}, we derive that $C^l(\Psi)$ must have belonged to   
$$
[b_\sigma\|\alpha_\sigma\|\tau\Specialks\|M,k\|L(\Psi)].
$$ 

Next, by hypothesis there exist no traces in  $C^l(\Psi)$; thus all of the $N$ traces in $\Psi^\sigma$ of $C^*(\Psi')$ must have arisen from contractions  $\Psi^k_a\Psi^\sigma_\ca$ and $\Specialks$. Furthermore, since the contractions of the form $\Specialks$ must all give rise to contractions $\Special$, all $N$ traces must have arisen from contractions of the form $\Psi^k_a\Psi^\sigma_\ca$. Thus $C^l(\Psi)$ belongs to 
\beq\label{NLPsi}
[N\Psi^k_a\Psi^\sigma_\ca\|b_\sigma\|
\tau\Specialks  \|M,k\|L(\Psi)],
\eeq
as claimed.

The above arguments show that the sublinear combination \eqref{NLPsi} contributes to $[M\Special\|\Delta^N\psi^\sigma\|\Local_k L(\Psi)]$ precisely as outlined in Claim \ref{last-claim}. This proves Claim \ref{last-claim}, and completes the proof of our Proposition.  
\end{proof}

\section{Proof of  Proposition  \ref{mainprop} (b): the case of order zero}
\label{chern_poly}

Recall from \S\ref{chern-poly} that the Chern polynomials are defined as invariant polynomials of the curvature form
$
\mathsf{R}_{a\cb}=R_{a\cb c\ce}dz^c\wedge dz^\ce.
$
We consider the variation of Chern polynomial under the perturbation of K\"ahler form $\omega_\epsilon=\omega_0+\epsilon\sqrt{-1}\pa\ol{\pa}\phi$ for a function $\phi\in C^\infty_0(\bC^n)$, where $\omega_0$ is the standard K\"ahler form.  The first variation of $\mathsf{R}_{a\cb}$ at $\epsilon=0$ is given by
$$
\mathsf{F}_{a\cb}=\phi_{ac\cb\ce}dz^c\wedge dz^\ce,
$$
where the indices on $\phi$ denote partial derivatives. Hence the $k$-th variation of the Chern character $\mathit{ch}_k(g_\epsilon)$ at $\epsilon=0$ is given by 
$$
\mathit{ch}_k(\phi)=\mathsf{F}_{a_1\ca_2}\wedge \mathsf{F}_{a_2\ca_3}\wedge\dots
\wedge \mathsf{F}_{a_k\ca_1},
$$
which we call the {\em Chern character of $\phi$ of order} $k$. This is exact since
$$
\Tch_k(\phi):=\phi_{a_1\ca_2\ce}dz^\ce\wedge \mathsf{F}_{a_2\ca_3}\wedge\dots
\wedge \mathsf{F}_{a_k\ca_1}
$$
gives
$$
d\big(\Tch_k(\phi)\big)=\ch_k(\phi).
$$
For any $k\in\bN$, consider any partition $k=k_1+k_2+\dots+k_m$ and define a $(k,k)$-from 
$$
\ch_{k_1,k_2,\dots,k_m}(\phi)=
\ch_{k_1}(\phi)\wedge\cdots\wedge
\ch_{k_m}(\phi).
$$
We  define a {\em $k$-Chern polynomial} by the complete contraction
$$
I_{k_1,k_2,\dots,k_m}(\phi)= \frac{1}{k!}\contr(
\ch_{k_1,k_2,\dots,k_m}(\phi)).
$$
Here we define the  contraction of $(k,k)$-from $T_{a_1\cb_1a_2\cb_2\dots a_k\cb_k}$
by $T_{a_1\ca_1a_2\ca_2\dots a_k\ca_k}$.
For example, 
$
I_{1}(\phi)=\Delta^2\phi,
$
$$
\begin{aligned}
I_{1,1}(\phi)&=
\Delta^2\phi\cdot\Delta^2\phi-\Delta \phi_{a\cb}\cdot
\Delta \phi_{b\ca},\\
I_{2}(\phi)&=\Delta \phi_{a\cb}\cdot
\Delta \phi_{b\ca}-
\phi_{a\cb c\cd}\cdot
\phi_{b\ca d\cc}.
\end{aligned}
$$
Since  $ch_k(\phi)$ is a compactly supported exact form, so is the product $\ch_{k_1,k_2,\dots,k_m}(\phi)$; hence we have
$$
\int I_{k_1,k_2,\dots,k_m}(\phi)=0
$$
with respect to the standard volume form on $\bC^n$.

\begin{definition}
For a factor in a complete contraction, its {\em height} is defined to be the number of  traces. The {\em height} of a complete contraction is defined to be  the sum of the height of all factors.  For a linear combination of complete contractions, the {\em maximum height}  is defined to be the maximum of the height of its complete contractions.
\end{definition}

\begin{lemma}
For each $k\ge2$,  the term in $I_k(\phi)$ with maximum height is 
$$
I^\text{\rm max}_k(\phi)
=
\Delta \phi_{a_1\ca_2}\Delta \phi_{a_2\ca_3}\cdots
\Delta \phi_{a_k\ca_1}.
$$
When $k=1$, the right-hand side is read as $\Delta \phi_{a\ca}=\Delta^2 \phi$.
\end{lemma}

\begin{proof}
Since the case $k=1$ is trivial, we let $k\ge2$. Recall that $I_k$ is computed as follows: skew symmetrize the tensor
$$
\phi_{a_1\ca_2 b_1\cc_1}\phi_{a_2\ca_3 b_2\cc_2}\cdots
\phi_{a_k\ca_1 b_k\cc_k} 
$$
over $b_1,\dots,b_k$ and $\cc_1,\dots,\cc_n$  and contract against $g^{b_1\cc_1}\cdots g^{b_n\cc_n}$\!.  The term of maximum height arises only when each pair $b_j\cc_j$ is contained in a factor. In that case the contraction gives
$
\Delta \phi_{a_1\ca_2}\cdots
\Delta \phi_{a_k\ca_1}.
$
Such terms arise $k!$ times, we obtain the claim.
\end{proof}

For the general case, the the terms of maximum height is given by
$$
I^\text{max}_{k_1,k_2,\dots,k_m}(\phi)=
I^\text{max}_{k_1}(\phi)I^\text{max}_{k_2}(\phi)\cdots
I^\text{max}_{k_m}(\phi).
$$
\begin{lemma}\label{lemma:internal.contr22}
Let $L(\phi)$ be a linear combination of terms of degree $\sigma$ of the form
$$
\contr(\pa^{(2,2)}\phi\otimes\cdots \otimes\pa^{(2,2)}\phi).
$$
Then there exists a linear combination of  $\sigma$-Chern polynomials $\Ch(\phi)$ such that
$$
L(\phi)=\Ch(\phi)+L^\sharp(\phi),
$$
where each term in $L^\sharp(\phi)$ has at least one factor without a trace.
\end{lemma}
\begin{proof}
Pick out the terms in $L(\psi)$ with  traces in each factor.  Among those, choose a term with a maximum height.  Such a term $C(\phi)$ can be written as a complete contraction of 
$$
\Delta\phi_{a\cb}\cdots\Delta\phi_{c\cd}.
$$
For such a term, we associate a directed graph as follows: Each factor defines a node and each contraction $\Delta\phi_{a\cb}\Delta\phi_{b\cc}$ defines an arc 
from the node $\Delta\phi_{a\cb}$ to the node $\Delta\phi_{b\cc}$. In particular, $\Delta\phi_{a\ca}=\Delta^2\phi$ defines a loop at the node $\Delta\phi_{a\ca}$.
Note that from each node there is exactly one outgoing arc and exactly one incoming arc.  Such a directed graph can be decomposed into simple cycles (i.e., a directed polygon), which corresponds to a complete contraction of the form
\beq\label{cycle}
\Delta\phi_{a_1\ca_2}\Delta\phi_{a_2\ca_3}\cdots
\Delta\phi_{a_m\ca_1}.
\eeq
When $m=1$, it is read as $\Delta\phi_{a_1\ca_1}=\Delta^2\phi$. Since the complete contraction \eqref{cycle} is a multiple of $I^\text{max}_m(\phi)$, we see that $C(\phi)$ is a multiple of $I^\text{max}_{k_1,k_2,\dots,k_m}(\phi)$ for a choice of $k_1,\dots,k_m$.

Thus subtracting $c\,I_{k_1,k_2,\dots,k_m}(\phi)$ from $L(\phi)$, we can remove the term without affecting other terms with the  maximum height (if they exists). By repeating this process, we derive  our claim. 
\end{proof}

\begin{lemma}\label{lemma:type22}
Let $L(\phi)$ be as in Lemma \ref{lemma:internal.contr22}.  Assume that each term in $L(\phi)$ has at least one factor of height $0$ (i.e., a factor without a trace).
If $L(\phi)$  integrates to zero in a dimension $n\ge\sigma-1$, then $L(\phi)=0$.
\end{lemma}

Observe that Proposition  \ref{mainprop} (b) follows  from the two lemmas above.

\begin{proof}
We polarize $L(\phi)$ and define $L(\Psi)=L(\psi^1,\dots,\psi^\sigma)$. Then $\int L(\Psi)=0$ holds. Assuming $L(\Psi)\ne0$ we will derive a contradiction.

Let $M$ be the maximum height of $L(\phi)$.  Picking out the terms of height $M$, we define the sublinear combination $[M\|L(\Psi)]$.  Now, for each $k\ge1$, we pick out the term of $[M\|L(\Psi)]$ with $k$ factors of height $0$; the sublinear combination is denoted by $ [k\|M\| L(\Psi)]$. Then, since each term has a factor of height $0$, we have
$$
[M\|L(\Psi)]=\sum_{k=1}^K [k\| M\|L(\Psi)]
$$
with $[K\|M\| L(\Psi)]\ne0$. Let $[\psi^1\|K\| M\| L(\Psi)]$ be the sublinear combination of  $[K\|M\| L(\Psi)]$ consisting of the terms for which $\psi^1$ has height $0$. Then we can recover $[K\|M\|L(\Psi)]$ from $[\psi^1\|K\|M\|L(\Psi)]$  by symmetrizing over $\psi^1,\dots,\psi^\sigma$. Thus we should have 
$
[\psi^1\|K\|M\| L(\Psi)]\ne0
$
and hence
$$
[\psi^1\|M\|L(\Psi)]:=\sum_{k=1}^K [\psi^1\|k\| M\|L(\Psi)]\ne0.
$$
Now we consider the local equation
\beq\label{IBPIpsi}
\Local_1L(\Psi)=0,
\eeq
which holds formally by the assumption $n\ge\sigma-1$. From $\Local_1L(\Psi)$ pick out the sublinear combination $[M+4\|\Local_1L(\Psi)]$ consisting of the terms of height $M+4$.

By the definition of $M$ and since all factors in all terms in $L(\Psi)$ have exactly four derivatives, each term in $[M+4\|\Local_1L(\Psi)]$ arises  from a term in 
$[\psi^1\|M\| L(\Psi)]$ by forcing each derivative in $\Psi^1_{ab\cc\cd}$ to hit the factor with the conjugate index (thus creating a trace in the latter). We then observe that we can formally  \emph{re-construct} $[\psi^1\|M\| L(\Psi)]$ from $[M+4\| \Local_1 L(\Psi)]$: In any term in the latter, by inspection of any  factor with more than four indices we can see the \emph{type} of index in $\Psi_1$ that gave rise to such a factor. Specifically, since all factors in any term in $L(\Psi)$ have type $(2,2)$, a factor $T^j$ in some term in $[M+4\| \Local_1 L(\Psi)]$,
$$
T^j:=\pa^{(p,q)}\Delta^r\psi^j
$$
can only have arisen  by applying derivatives of type $(r+p-2,r+q-2)$  to the factor 
$$
S^j:=\pa^{(2-s,2-s)}\Delta^s\psi^j,\quad
s=4-r-p-q,
$$
in a term in $[\psi^1\|M\| L(\Psi)]$. Each such application of a derivative creates a trace $\Delta$. Since the correspondence $T^j\mapsto S^j$ is one to one,
the equation
$$
[M+4\|\Local_1L(\Psi)]=0,
$$
which follows from \eqref{IBPIpsi} and holds \emph{formally}, forces $[\psi^1\|M\| L(\Psi)]=0$. This is a contradiction.
\end{proof}

\section{Appendix}
\subsection{Calculation of the Bergman kernel of line bundles}
\label{line-bundle-sec} 
We give an algebraic procedure for computing the asymptotic expansion of the Bergman kernel for the positive line bundles.  Now there are  several effective procedures for computing the expansion.  Here is a part of the list:
\begin{enumerate}
\item  Tian's peak section method: \cite{Ti}, \cite{Lu}, \cite{LT}.

\item  Use the asymptotic expansion of a Laplace integral: \cite{En}, \cite{Lo}, \cite{Xu}.

\item Use the Szeg\"o kernel of the tube in a line bundle:
\cite{Ze}, \cite{Ca}, \cite{BBS}.

\end{enumerate}
\noindent
Our method is in the line of (3).  The main tool is the infinite order microdifferential operators that was introduced by Boutet de Monvel \cite{BdM}; some explicit calculations of the Bergman and Szeg\"o kernels were done in \cite{HKN1} and \cite{HKN2}.  However,  these papers contain technical details and it is not easy to find the formula needed in the present situation.  The following is a user's guide of \cite{HKN1} specialized to the case of the tube in a line bundle --- this is a revision of  second author's unpublished notes in 1999 that had been referred by Mabuchi \cite{Mab}.

Let $(X,\calL)$ be a polarized manifold with real analytic hermitian metric.  Then the unit disc bundle in the dual bundle $\calL^*$,
$$ 
\Omega=\{v\in \calL^*: \|v\|<1\},
$$ 
is a strictly pseudoconvex domain with a real analytic defining function $\rho(v,\ol v)=-\log\|v\|^2$, which is positive inside. Take a local holomorphic section  $\varphi$ of $\calL^*$ and set $h(z,\ol z)=\|\varphi(z)\|^2$. Then in the local coordinates $(\lambda,z)=(\lambda,z^1,\dots, z^n)$, given by $v=\lambda \varphi(z)$, one gets
$$
\rho(v,\ol v)=\log \lambda+\log\ol\lambda+\log h(z,\ol z).
$$
On the boundary $M=\pa\Omega$, we define contact form $\theta=-\sqrt{-1}\pa\rho|_M$ so that $d\theta$ agrees with the pullback of the K\"ahler from $\omega=\sqrt{-1}\pa\ol\pa\rho$.

Let $S(v,\ol v)$ be the Szeg\"o kernel of $M$ with respect to the $L^2$-inner product given by the volume form $\theta\wedge (d\theta)^n$. We  know from \cite{BS} that the Szeg\"o kernel (on the diagonal) is written as the Laplace integral
$$
S(v,\ol v)=\int_0^\infty e^{-t\rho(v,\ol v)} b(z,t)dt
$$
of a classical symbol $b(z,t)\in S^n(X\times\bR_+)$ with asymptotic expansion
$$
b(z,t)\sim \sum_{j=0}^{\infty}b_j(z)t^{n-j-1}\quad\text{as }t\to\infty.
$$ 
The Bergman kernel for  $H^0(X,\calL^m)$ is then given by the Fourier coefficient on the fibers 
$$
B_m(z)=\frac1{2\pi }\int_0^{2\pi} e^{-\sqrt{-1}m\phi}S(e^{i\phi}v,\ol v)d\phi.
$$
Thus noting  $\rho(e^{\sqrt{-1}\phi}v,\ol v)=-\sqrt{-1}\phi$ for small $\phi$, we obtain
$$
B_m(z)\sim b(z,m)\quad\text{as }m\to\infty
$$
by  the method of stationary phase as in \cite{Ze}. Therefore $a_j(z)=b_j(z)$ and we are reduced to compute the symbol $b(z,t)$.

 For a point $p\in X$, we take Bochner's coordinates $z$ around $p$ in which the K\"ahler potential is written in the form
\beq\label{Kpotential}
\log h=|z|^2+H(z,\ol z),\quad
H(z,\ol z)=\sum_{|\alpha|,|\beta|\ge2}a_{\alpha\ol\beta}z^\alpha\ol z^\beta.
\eeq
Then, setting $z_0=\log\lambda$, we may write  $M$  as
$$
\rho=z_0+\ol z_0+|z|^2+H(z,\ol z).
$$
Using
$\omega_g^n=c \det(g_{j\ol k})dz\wedge d\ol z$, where $c$ is a constant, we have
$$
\begin{aligned}
d\rho\wedge \theta\wedge(d\theta)^n
&= dz_0\wedge d\ol z_0\wedge\pi^*\omega_g^n \\
&=c \det(g_{j\ol
k})dz_0\wedge d\ol z_0\wedge dz\wedge d\ol z.
\end{aligned}
$$ 
Therefore the integration with respect to $\theta\wedge(d\theta)^n$ corresponds to the $\delta$-function $c(\det g)\,\delta(\rho)$ in the coordinates $(z_0,z)$. Note that $\det g=\det(I_n+\pa\ol\pa H)$ and the complexification is given by $\det g(z,\ol z)=\det(I_n+\pa_z\ol\pa_z H(z,\ol z))$. Let $(t,\zeta)=(t,\zeta^1,\dots,\zeta^n)\in \bC^{n+1}$ be the fiber coordinate of $T^*\bC^{n+1}$ with respect to the coframe  $dz^0,dz^1,\dots,dz^n$.  Substituting $\zeta/t$ into $\ol z$, we define
$$
A(z;\zeta,t)=(\det g)(z,\zeta/t)\exp[-H(z,\zeta/t)t].
$$
Then $ A(z;\zeta,t)$ has Laurent series expansion in $t$:
$$
A(z;\zeta,t)=1+\sum_{j=1}^\infty P_j(z,\zeta)t^{-j},
$$
where $P_j(z,\zeta)$ is holomorphic in $z$ (near $z=0$) and polynomial in $\zeta$.  By replacing $\zeta$ by the partial derivative $\pa_z$, we may define a differential operator $A(z;\pa_z, t)$ in $z$ with a parameter $t$. The (formal) adjoint of $A$ is defined by 
$$
A^*=1+\sum_{j=1}^\infty P_j^*(z,\pa_z)(-t)^{-j},
$$
where $ P_j^*(z,\pa_z)$ is the formal adjoint of $P_j(z,\pa_z)$ with respect to the standard metric on $\bC^n$; see the paragraph below Theorem \ref{A-thm} for the reason why we replace $t$ by $-t$.
The inverse of $A^*$ is given by the Neumann series
$$
(A^*)^{-1}=\sum_{l=0}^\infty\Big(-\sum_{j=1}^\infty P_j^*(z,\pa_z)(-t)^{-j}\Big)^l,
$$
which is well-defined as a  Laurent series  of $t$ with coefficients in the ring of differential operators in $z$. If we write
$$
(A^*)^{-1}=\sum_{j=0}^\infty Q_j(z,\pa_z)t^{-j},
$$
then  we have
$$
\begin{aligned}
Q_1&=P_1^*,\\ 
Q_2&=-P_2^*+(P_1^*)^2,
\\
Q_3&=P_3^*-P_1^*P_2^*-P_2^* P_1^*+(P_1^*)^3
\end{aligned}
$$
and so on.
We now take the constant term of $Q_j(z,\pa_z)$ at $z=0$, which we denote by $Q_j(0)$.  Then we obtain a Laurent series  in $t$
$$
(A^*)^{-1}(0;0,t)=\sum_{j=0}^\infty Q_j(0)t^{-j}.
$$ 
From this construction, it is clear that $Q_j(0)$ is a local invariant of $H$ in the sense of \S2.

\begin{theorem}\label{A-thm}  Take Bochner's coordinates around $p$ and write the K\"ahler potential as \eqref{Kpotential}. Then, at $p$, one has
\beq\label{bAeq}
b(p,t)=(A^*)^{-1}(0;0,t)t^n.
\eeq
In particular the coefficients in expansion of $B_m(p)$ are given by $a_j(p)=Q_j(0)$.
\end{theorem}

This theorem is a special case of \cite[Lemma 2.2]{HKN1}. The calculus above is done in the ring of microdifferential  operators, in which the negative powers of $t$ are regarded as microdifferential operators of negative order  $\pa_{z^0}^{-j}$; this is why we replace $t$ by $-t$ when we take the formal adjoint. See also \cite{HKN2} for more detailed technique of the calculation.

As an example  we compute the linear terms in $a_j$ and confirm \eqref{ajLu}. If we take the first variation of the determinant, we get
$$
\det g=1+\Delta H+O(H^2),
$$
where $\Delta=\pa_{z^a}\pa_{\ol z^a}$ is the flat K\"ahler Laplacian. Thus 
$$
\begin{aligned}
A(z;\pa_z,t)&=\det g(z,t^{-1}\pa_z)\cdot \exp (-H(z,t^{-1}\pa_z)t)\\
&=1+\Delta H(z,t^{-1}\pa_z)-H(z,t^{-1}\pa_z)t+O(H^2).
\end{aligned}
$$
The formal adjoint is given by
$$
A^*=1+H(z,-t^{-1}\pa_z)t+\sum_{j=2}^\infty \frac{j}{(j+1)!}\Delta^{j+1}H(z,-t^{-1}\pa_z)t^{-j}+O(H^2).
$$
Here $-t^{-1}\pa_z$ is substituted after applying powers of  $\Delta$ to $H$. Thus the inverse is
$$
(A^*)^{-1}(0;0;t)=1-H(0,0)t-\sum_{j=2}^\infty \frac{j}{(j+1)!}\Delta^{j+1}H(0,0)t^{-j}+O(H^2).
$$
Using $H(0,0)=0$ and $\Delta^{j+1}H(0,0)=-\Delta^{j-1}_g S_g(p)+O(R^2)$, we get
$$
(A^*)^{-1}(0;0,t)=1+\sum_{j=1}^\infty \frac{j}{(j+1)!}\Delta_g^{j-1}S_g(p)t^{-j}+O(R^2),
$$
which is equivalent to \eqref{ajLu}.

It is not difficult to write $a_j$ in terms of the derivatives of $H$. But rewriting invariants of $H$ in terms of the covariant derivatives of the curvature requires extra effort.  Fortunately, Xu \cite{Xu} has give an efficient way to do it by using directed graphs associated to K\"ahler invariants.  Combining the work of Xu with the algorithmic proof of Theorem \ref{A-thm} provides an effective method of expressing $a_j(g)$ as a divergence plus a Chern polynomial.%

\subsection{$Q$-curvature on CR manifolds}
We here give an application of the main theorem to the $Q$-curvature in CR geometry.  We start by quickly recalling the basic properties of the $Q$-curvature 
by following \cite{FH}. Let $M$ be a  strictly pseudoconvex CR manifold of dimension $2n+1$; we denote the holomorphic tangent bundle by $T^{1,0}M$, which we assume to be integrable.  With a choice of contact form $\theta$, we can define the Levi metric on $T^{1,0}M$ by the two form $d\theta$.  In analogy with the hermitian connection on complex manifolds, one can define a canonical connection $\nabla$ on $TM$, called Tanaka-Webster connection, that preserves the subbundle  $T^{1,0}M$  and the Levi metric.  The $Q$-curvature of CR manifold $Q_\theta$ with respect to $\theta$ is a local invariant of the Levi metric $d\theta$, which can be expressed in terms of the curvature, torsion and their covariant derivatives of $\nabla$.  While $Q_\theta$ is not a CR invariant, there is a transformation law under the scaling of the contact form $\wh\theta=e^\Upsilon \theta$:
$$
e^{(n+1)\Upsilon}Q_{\wh\theta}=Q_{\theta}+P_\theta\upsilon,
$$
where $P_\theta$ is a self-adjoint CR invariant operator of order $2n+2$ such that $P_\theta1=0$.  Here the CR invariant operator means that $ e^{(n+1)\Upsilon}P_{\wh\theta}=P_\theta$. In particular, we see that the total $Q$-curvature
$$
\int_M Q_\theta \theta\wedge (d\theta)^n
$$
is independent of $\theta$ and defines a global invariant of CR manifolds. In dimension $3$, we have 
$$
Q_\theta=\frac23(\Delta_b S-2\Im \nabla_{\ca\cb}A_{ab}),
$$
which is a divergence and integrates to zero on compact manifolds. However, for higher dimensions, it is not easy to write down $Q_\theta$ explicitly.

On the other hand, if an open set of $M$ is embedded in $\bC^{n+1}$, then we can choose $\theta$ so that $ Q_\theta=0$ on that set.   Then, for any $\wh\theta=e^\Upsilon\theta$, the transformation law gives
$$
Q_{\wh\theta}=P_{\wh\theta}\Upsilon.
$$
The right-hand side  can be written as a divergence of one form depending on $\Upsilon$, but it is not clear the one form is a local invariant of $\wh\theta$.  A natural guess is the following:

\begin{conjecture}  On CR manifolds, the $Q$-curvature is a divergence, i.e., there exist form-valued local invariants $T_a$ and $T_\ca$ of $\theta$ such that $Q_\theta=\nabla_\ca T_a+\nabla_a T_\ca$.
\end{conjecture}

Here $T_a$ and $T_\ca$ are respectively sections of $(T^{1,0}M)^*$ and $(T^{0,1}M)^*$; the contraction is taken with respect to the Levi metric for $\theta$.

We confirm the conjecture in the special case when $M=\pa\Omega$, the boundary of the unit disc bundle in a negative line bundle $\calL^*$ used in \S\ref{line-bundle-sec}. For the standard contact form $\theta=-i\pa\rho|_M$, the $Q$-curvature becomes $S^1$-invariant and  $Q_\theta$ can be seen as a function on the base manifold $X$, which we denote by $Q_\omega$.  Moreover, by the construction, $Q_\omega$ is shown to be a local invariant of the K\"ahler form $\omega=d\theta$.

\begin{theorem}
There exist one form valued local K\"ahler invariants $T_a$ and $T_\ca$ such that
$Q_\omega=\nabla_\ca T_a+\nabla_a T_\ca$.
\end{theorem}

\begin{proof}
By the scaling of contact from by constant $\wh\theta=e^c\theta$, we have $Q_{\wh\theta}=e^{-(n+1)c}Q_{\theta}$.  Thus $ Q_{\omega}$ has geometric weight $n+1$. We next show that the integral
$$
\int_M Q_\theta \theta\wedge(d\theta)^n=2 \pi\int_X Q_\omega\omega^n
$$
depends only on the K\"ahler class. For an $f\in C^\infty(X)$, we define a family of fiber metrics on $\calL^*$ by $h_\epsilon=e^{\epsilon f}h$ and set $M_\epsilon$ to be the unit circle bundle for $h_\epsilon$.  Let $Q_\epsilon$ be the $Q$-curvature for $M_\epsilon$ with the standard contact form $\theta_\epsilon$ defined from $h_\epsilon$. Then the curvature of $h_\epsilon$ is given by
\begin{equation}\label{pert-omega}
\omega_\epsilon=\omega+\epsilon\sqrt{-1}\pa\overline{\pa}f,
\end{equation}
which is positive if $\epsilon$ is small. By Matsumoto \cite{Mat}, we know that the total $Q$-curvature is invariant under the deformation of integrable CR structures. Thus the integral of $Q_\epsilon$ is independent of $\epsilon$ and hence the integral of $Q_\omega$ depends only on the K\"ahler class. Therefore $Q_\omega$ satisfies the assumption of the main theorem and we obtain a decomposition
$$
Q_\omega=\Ch_{n+1}(g)+\nabla_\ca T_a+\nabla_a T_\ca,
$$
where $\Ch_{n+1}(g)$ is a Chern polynomial of degree $n+1$.  As $X$ has dimension $n$,  $\Ch_{n+1}(g)$ vanishes identically and we get the desired expression of $Q_\omega$.
\end{proof}

This theorem can be generalized to Sasakian manifolds.  Recall that a Sasakian manifold is a CR manifold with a contact form $\theta$ for which the Tanaka-Webster torsion vanishes.  For  such $\theta$, the expression of $Q_\theta$ in terms Tanaka-Webster connection and its covariant derivatives agrees with the one for curvature for a K\"ahler manifold and its derivatives; see \cite{Web}.   In particular, we see that the total $Q$-curvature of a compact Sasakian manifold vanishes.

\subsection*{Acknowledgments}
SA was partially supported by NSERC grants 488916 and 489103, and Clay and Sloan fellowships during the earlier stages of this project. KH was partially supported by JSPS KAKENHI grant 60218790. Part of the work was completed during visits by both authors to the ANU in Australia, BIRS Canada, and the  CRM, UAB Barcelona, and by visits of KH to the University of Toronto and SA to the University of Tokyo. We thank all these institutions for their hospitality.

\end{document}